\documentclass[english]{article}
\usepackage{mathptmx}
\usepackage[T1]{fontenc}
\usepackage[latin9]{inputenc}
\usepackage{color}
\usepackage{amsmath}
\usepackage{amsthm}
\usepackage{amssymb}

\makeatletter
\usepackage{amsthm}

\usepackage{amsfonts}
\usepackage{enumerate}
\usepackage{fancyhdr}
\usepackage{xcolor}

\theoremstyle{plain}
\theoremstyle{remark}

\usepackage{amsthm}\usepackage{fullpage}\usepackage{makeidx}\usepackage{amsfonts}\usepackage{latexsym}
\numberwithin{equation}{section}

\numberwithin{figure}{section}

\usepackage{fullpage}\usepackage{makeidx}\usepackage{amsfonts}\usepackage{latexsym}

\usepackage{fullpage}\usepackage{makeidx}\usepackage{amsfonts}\usepackage{latexsym}

\theoremstyle{definition}
\newtheorem{lemma}{Lemma}[section]
\newtheorem{theorem}[lemma]{Theorem}\newtheorem{corollary}[lemma]{Corollary}\newtheorem{proposition}[lemma]{Proposition}\newtheorem{definition}[lemma]{Definition}\newtheorem{remark}[lemma]{Remark}\usepackage{times}

\usepackage[blocks]{authblk}% The option is for block layout
\title{Permutation orbifolds of vertex operator superalgebra and associative algebras}

\author{Chongying Dong}
\affil{Department of Mathematics, University of
California at Santa Cruz, Santa Cruz, CA {\rm 95064}, USA}

\author{Feng Xu }

\affil{Department of Mathematics, University of California at Riverside, Riverside, CA {\rm 92521}, USA }

\author{Nina Yu\footnote{Supported by National Natural Science Foundation of China  11971396, 12131018 and 12161141001}}
\affil{School of Mathematical Sciences, Xiamen University, Xiamen, Fujian {\rm 361005}, CHINA}

\usepackage{babel}

\makeatother

\usepackage{babel}
\begin{document}
\maketitle
\begin{abstract}
Let $V$ be a vertex operator superalgebra and $g=\left(1\ 2\ \cdots k\right)$
be a $k$-cycle which is viewed as an automorphism of the tensor product
vertex operator superalgebra $V^{\otimes k}$. In this paper, we construct
an explicit isomorphism from $A_{g}\left(V^{\otimes k}\right)$ to
$A\left(V\right)$ if $k$ is odd and to $A_{\sigma}\left(V\right)$
if $k$ is even where $\sigma$ is the canonical automorphism of $V$
of order 2 determined by the superspace structure of $V.$ These recover
previous results by Barron and Barron-Werf that there is a one-to-one
correspondence between irreducible $g$-twisted $V^{\otimes k}$-modules
and irreducible $V$-modules (resp. irreducible $\sigma$-twisted
$V$-modules) when $k$ is odd (resp. even). This explicit isomorphism
is expected to be useful in our further study on the Zhu algebra of
fixed point subalgebra.
\end{abstract}

\section{Introduction}

Let $V$ be a vertex operator superalgebra and $G$ be a finite automorphism
group of $V$. Then the space of $G$-invariants $V^{G}$ itself is
also a vertex operator superalgebra. \emph{Orbifold theory} studies
$V^{G}$ and its representation theory. The key ingredients in the
study of orbifold theory are the $g$-twisted $V$-modules for $g\in G$.
How to construct twisted modules in general is still a challenging
problem. Let $k$ be a fixed positive integer. The permutation orbifolds
study the representations of the tensor product vertex operator (super)algebra
$V^{\otimes k}$ with the natural action of the symmetric group $S_{k}$
as an automorphism group. A systematic study of permutation orbifolds
in the context of vertex operator algebras was started in \cite{BDM},
where a connection between twisted modules for tensor product vertex
operator algebra $V^{\otimes k}$ with respect to permutation automorphisms
and $V$-modules was found. Barron extended these results to $g$-twisted
$V^{\otimes k}$-modules for a vertex operator superalgebra $V$ and
$g=\left(1\ 2\cdots k\right)$ with $k$ odd. It was shown in \cite{Ba}
that the categories of weak, weak admissible and ordinary $g$-twisted
$V^{\otimes k}$-modules are isomorphic to the categories of weak,
weak admissible and ordinary $V$-modules, respectively. Moreover,
if $k$ is even, the category of weak $g$-twisted $V^{\otimes k}$-modules
is isomorphic to the category of $\sigma$-twisted $V$-modules \cite{BW}
where $\sigma$ is the canonical automorphism of $V$ of order $2$
arising from the structure of superspace. The $\sigma$-twisted $V$-modules
were called parity-twisted modules in literatures in \cite{BW}.

For a vertex operator algebra $V$, an associative algebra $A\left(V\right)$
was constructed and studied in \cite{Z} such that irreducible $V$-modules
are in one-to-one correspondence with irreducible $A\left(V\right)$-modules.
This work was extended to vertex operator superalgebras in \cite{KW}.
On the other hand, Zhu's work was extended in another direction. For
any finite order automorphism $g$ of $V$, an associative algebra
$A_{g}\left(V\right)$ was constructed such that irreducible $g$-twisted
$V$-modules are in one-to-one correspondence with irreducible $A_{g}\left(V\right)$-modules
\cite{DLM1}. For any vertex operator superalgebra $V$ and for any
finite order automorphism $g$ of $V$, an associative algebra $A_{g}\left(V\right)$
was constructed in \cite{X}. For any $g$-twisted $V$-module, $A_{g}\left(V\right)$-module
was associated and conversely, for any $A_{g}\left(V\right)$-module
a $g$-twisted $V$-module was constructed. Furthermore, there is
one-to-one correspondence between irreducible $g$-twisted $V$-modules
and irreducible $A_{g}\left(V\right)$-modules, and if $V$ is $g$-rational,
then $A_{g}\left(V\right)$ is semisimple and there are finitely many
irreducible $g$-twisted $V$-modules up to isomorphism \cite{DZ}.
In our forthcoming papers, we shall use this explicit isomorphism
to compute the Zhu algebra of fixed point subalgebra.

Let $V$ be a vertex operator superalgebra and $g=\left(1\ 2\cdots k\right)\in S_{k}$
be an automorphism of $V^{\otimes k}.$ From \cite{DZ,Ba,BW} we know
that $A\left(V\right)$ should be isomorphic to $A_{g}\left(V^{\otimes k}\right)$
\emph{(resp. $A_{\sigma}\left(V\right)$)} if $k$ is odd \emph{(resp.
even)}. Our goal in this paper is to give an explicit isomorphism
from $A_{g}\left(V^{\otimes k}\right)$ to $A\left(V\right)$ \emph{(resp.
$A_{\sigma}\left(V\right)$)} when $k$ is odd \emph{(resp. even)}.
Consequently, we give a new proof that there is a one-to-one correspondence
between irreducible $g$-twisted $V^{\otimes k}$-modules and irreducible
$V$-modules \emph{(resp. $\sigma$-twisted $V$-modules}) when $k$
is odd\emph{ (resp. even)} \cite{DZ,Ba,BW}.

The paper is organized as follows: In Section 2, we recall some basic
notions and properties in vertex operator superalgebra theory. In
Section 3, we recall the associative algebra $A_{g}\left(V\right)$
associated to a vertex operator superalgebra $V$. In Section 4, we
first recall properties of permutation orbifold vertex operator superalgebra
and $\Delta_{k}$-operator. Then we use the $\Delta_{k}$-operator
to construct an explicit isomorphism between $A_{g}\left(V^{\otimes k}\right)$
and $A\left(V\right)$ \emph{(resp. $A_{\sigma}\left(V\right))$}
when $k$ is odd (\emph{resp. even}).

\section{\label{sec:Basics}Basics}

In this section we review some basics of vertex operator superalgebras
\cite{Bo,FLM,DL}. Throughout this paper, $z_{0},z_{1},z_{2},$ etc.
are independent commuting formal variables.

\emph{A super vector space} is a $\mathbb{Z}_{2}$-graded vector space
$V=V_{\bar{0}}\oplus V_{\bar{1}}.$ The elements in $V_{\bar{0}}$
(resp. $V_{\bar{1}}$) are called even (resp. odd). Let $\left|v\right|$
be $0$ if $v\in V_{\bar{0}}$ and $1$ if $v\in V_{\bar{1}}.$

\begin{definition} A \emph{vertex operator superalgebra} is a $\frac{1}{2}\mathbb{Z}_{+}$-graded
super vector space $V=V_{\bar{0}}\oplus V_{\bar{1}}$ with $V_{\bar{0}}=\sum_{n\in\mathbb{Z}}V_{n}$
and $V_{\bar{1}}=\sum_{n\in\frac{1}{2}+\mathbb{Z}}V_{n}$ satisfying
$\dim V_{n}<\infty$ for all $n$ and $V_{m}=0$ if $m$ is sufficiently
small. And $V$ is equipped with a linear map
\begin{align*}
 & V\to\text{\ensuremath{\left(\text{End}V\right)\left[\left[z,z^{-1}\right]\right]},}\\
 & v\mapsto Y\left(v,z\right)=\sum_{n\in\mathbb{Z}}v_{n}z^{-n-1}\ (v_{n}\in\text{End}V),
\end{align*}
and with two distinguished vector ${\bf 1}\in V_{0}$, $\omega\in V_{2}$
satisfying the following conditions: For $u,v\in V,$ and $m,n\in\mathbb{Z},i\in\frac{1}{2}\mathbb{Z},$
\begin{alignat*}{1}
 & u_{n}v=0\ \text{for}\ n\ \text{sufficiently\ large;}\\
 & Y\left({\bf 1},z\right)=\text{Id}_{V};\\
 & Y\left(v,z\right){\bf 1}\in V\left[\left[z\right]\right]\ \text{and}\ \lim_{z\to0}Y\left(v,z\right){\bf 1}=v;\\
 & \left[L\left(m\right),L\left(n\right)\right]=\left(m-n\right)L\left(m+n\right)+\frac{1}{12}\left(m^{3}-m\right)\delta_{m+n,}c;\\
 & \frac{d}{dz}Y\left(v,z\right)=Y\left(L\left(-1\right)v,z\right);\\
 & L\left(0\right)|_{V_{i}}=i\
\end{alignat*}
where $Y\left(\omega,z\right)=\sum_{n\in\mathbb{Z}}L\left(n\right)z^{-n-2}$
and the Jocobi identity holds:
\begin{gather*}
z_{0}^{-1}\left(\frac{z_{1}-z_{2}}{z_{0}}\right)Y\left(u,z_{1}\right)Y\left(v,z_{2}\right)-\left(-1\right)^{\left|u\right|\left|v\right|}z_{0}^{-1}\delta\left(\frac{z_{2}-z_{1}}{-z_{0}}\right)Y\left(v,z_{2}\right)Y\left(u,z_{1}\right)\\
=z_{2}^{-1}\delta\left(\frac{z_{1}-z_{0}}{z_{2}}\right)Y\left(Y\left(u,z_{0}\right)v,z_{2}\right),
\end{gather*}
where $\delta\left(z\right)=\sum_{n\in\mathbb{Z}}z^{n}$ and $\left(z_{i}-z_{j}\right)^{n}$
is expanded as a formal power series in $z_{j}$.

\end{definition}

In the following, we will denote such a vertex operator superalgebra
by $V=\left(V,Y,\textbf{1},\omega\right)$. If $V_{\bar{1}}=0$, $V$
is a vertex operator algebra.

\begin{remark} \label{tensor product VOSA} If $\left(V_{i},Y_{i},\textbf{1}_{i},\omega_{i}\right)$,
$1\le i\le k$ are vertex operator superalgebras, let $V=V_{1}\otimes V_{2}\otimes\cdots\otimes V_{k}$,
$\mathbf{1}={\bf 1}_{1}\otimes{\bf 1}_{2}\otimes\cdots\otimes{\bf 1}_{k}$
, $\omega=\omega_{1}\otimes\mathbf{1}_{2}\otimes\cdots\otimes\mathbf{1}_{k}+{\bf 1}_{1}\otimes\omega_{2}\otimes{\bf 1}_{3}\otimes\cdots\otimes{\bf 1}_{k}+{\bf 1}_{1}\otimes\cdots\otimes{\bf 1}_{k-1}\otimes\omega_{k},$
and set
\begin{alignat}{1}
 & Y\left(u_{1}\otimes u_{2}\otimes\cdots\otimes u_{k},z\right)\left(v_{1}\otimes v_{2}\otimes\cdots\otimes v_{k}\right)\nonumber \\
= & \left(-1\right)^{\sum_{i=2}^{k}\left|u_{i}\right|\left(\left|v_{1}\right|+\cdots+\left|v_{i-1}\right|\right)}Y_{1}\left(u_{1},z\right)v_{1}\otimes Y_{2}\left(u_{2},z\right)v_{2}\otimes\cdots\otimes Y_{k}\left(u_{k},z\right)v_{k}\label{tensor product-super sign}
\end{alignat}
for $u_{i},v_{i}\in V_{i},$ ~$1\le i\le k$. Then $\left(V,Y,\mathbf{1},\omega\right)$
is a vertex operator superalgebra, which is called a \textit{tensor
product vertex operator superalgebra}. \end{remark}

\begin{definition} Let $V$ be a vertex operator superalgebra. An
\emph{automorphism} $g$ of $V$ is a linear automorphism of $V$
preserving $\omega$ such that $gY\left(v,z\right)g^{-1}=Y\left(gv,z\right)$
for $v\in V$. \end{definition} Note that any automorphism of $V$
preserves each homogeneous space $V_{n}$ and hence preserves $V_{\bar{0}}$
and $V_{\bar{1}}.$ Denote by $\text{Aut}\left(V\right)$ the group
of automorphisms of $V$. Let $g\in\text{Aut}\left(V\right)$ be an
automorphism of order $k<\infty$. Then $V$ is a direct sum of the
eigenspaces $V^{j}$ of $g$,
\[
V=\oplus_{j\in\mathbb{Z}/k\mathbb{Z}}V^{j},
\]
where $V^{j}=\left\{ v\in V\mid gv=\eta^{j}v\right\} $ for $\eta$
a fixed primitive $k$-th root of unity.

\begin{definition} Let $\left(V,Y,\textbf{1},\omega\right)$ be a
vertex operator superalgebra. A \emph{weak g-twisted $V$-module }is
a $\mathbb{Z}_{2}$-vector space $M=M^{\bar{0}}\oplus M^{\bar{1}}$
equipped with a linear map
\begin{align*}
Y_{g}(\cdot,z):\  & V\to({\rm End}\;M)\left[\left[z^{1/k},z^{-1/k}\right]\right]\\
 & v\mapsto Y_{g}\left(v,z\right)=\sum_{n\in\frac{1}{k}\mathbb{Z}}v_{n}^{g}z^{-n-1},
\end{align*}
with $v_{n}^{g}\in\left({\rm End}\;M\right)^{\left|v\right|}$ such
that for $u,v\in V$ and $w\in M$ the following conditions hold:

(1) For any $v\in V,\ w\in M$, $v_{n}^{g}w=0$ for $n$ sufficiently
large;

(2) $Y_{g}({\bf 1},z)=\text{Id}_{M}$ (the identity operator on $M$);

(3) For $\mathbb{Z}_{2}$-homogeneous $u,v\in V$, the following Jacobi
identity holds
\begin{gather*}
z_{0}^{-1}\left(\frac{z_{1}-z_{2}}{z_{0}}\right)Y_{g}(u,z_{1})Y_{g}(v,z_{2})-(-1)^{|u||v|}z_{0}^{-1}\delta\left(\frac{z_{2}-z_{1}}{-z_{0}}\right)Y_{g}(v,z_{2})Y_{g}(u,z_{1})\\
=\frac{z_{2}^{-1}}{k}\sum_{j\in\mathbb{Z}/k\mathbb{Z}}\delta\left(\eta^{j}\frac{\left(z_{1}-z_{0}\right)^{1/k}}{z_{2}^{1/k}}\right)Y_{g}\left(Y\left(g^{j}u,z_{0}\right)v,z_{2}\right).
\end{gather*}
\end{definition}

Note that the above generalized twisted Jacobi identity is equivalent
to
\begin{gather*}
z_{0}^{-1}\left(\frac{z_{1}-z_{2}}{z_{0}}\right)Y_{g}(u,z_{1})Y_{g}(v,z_{2})-(-1)^{|u||v|}z_{0}^{-1}\delta\left(\frac{z_{2}-z_{1}}{-z_{0}}\right)Y_{g}(v,z_{2})Y_{g}(u,z_{1})\\
=z_{2}^{-1}\delta\left(\frac{z_{1}-z_{0}}{z_{2}}\right)^{-r/k}\delta\left(\frac{z_{1}-z_{0}}{z_{2}}\right)Y_{g}\left(Y\left(u,z_{0}\right)v,z_{2}\right)
\end{gather*}
for $u\in V^{r},$ $r=0,\cdots,k-1.$ This implies that for $v\in V^{r}$,
\[
Y_{g}\left(v,z\right)=\sum_{n\in r/k+\mathbb{Z}}v_{n}^{g}z^{-n-1}
\]
and for $v\in V,$
\[
Y_{g}\left(gv,z\right)=\lim_{z^{1/k}\to\eta^{-1}z^{1/k}}Y_{g}\left(v,z\right).
\]

A weak $1$-twisted $V$-module is called a \emph{weak $V$-module}.

\begin{definition} An\emph{ admissible $g$-twisted $V$-module }is
a weak $g$-twisted $V$-module $M$ equipped with a $\frac{1}{2k}\mathbb{N}$-grading
$M=\oplus_{n\in\frac{1}{2k}\mathbb{N}}M(n)$ such that for homogeneous
$v\in V,$ $n\in\frac{1}{2k}\mathbb{N}$ and $m\in\frac{1}{k}\mathbb{Z}$,
\[
v_{m}^{g}M(n)\subset M(n+\text{wt}v-m-1).
\]
\end{definition}

One may assume $M\left(0\right)\not=0$. An admissible $1$-twisted
$V$-module is called an\emph{ admissible $V$-module}.

\begin{definition} A \emph{$g$-twisted $V$-module } is a weak $g$-twisted
$V$-module $M$ satisfying the conditions that $M=\oplus_{h\in\mathbb{C}}M_{(h)}$,
where
\[
M_{(h)}=\left\{ w\in M\mid L(0)^{g}w=hw\right\} ,\text{for}\ L\left(0\right)^{g}=\omega_{1}^{g},
\]
$\dim M_{(h)}<\infty$ for all $h\in\mathbb{C}$, and $M_{(h+\frac{n}{2k})}=0$
for fixed $h\in\mathbb{C}$ and for all sufficiently small integers
$n$.\end{definition}

A $1$-twisted $V$-module is called a \emph{$V$-module}.

\section{The associative algebras}

In this section, we recall the associative algebras associated to
a vertex operator superalgebra.

Define an automorphism $\sigma$ of $V$ such that $\sigma|V_{\bar{0}}=1$
and $\sigma|V_{\bar{1}}=-1$. Then $\sigma$ is a central element
of $\text{Aut}\left(V\right).$ Fix $g\in\text{Aut}\left(V\right)$
of order $T_{0}$. Let $o\left(g\sigma\right)=T$. Then $V$ can be
decomposed into sum of eigenspaces with respect to the actions of
$g$ and $g\sigma$ as follows:
\begin{alignat*}{1}
 & V=\oplus_{r\in\mathbb{Z}/T_{0}Z}V^{r}\\
 & V=\oplus_{r\in\mathbb{Z}/TZ}V^{r*}
\end{alignat*}
where $V^{r}=\left\{ v\in V\mid gv=e^{2\pi ir/T_{0}}v\right\} $ and
$V^{r\ast}=\left\{ v\in V\mid g\sigma v=e^{2\pi ir/T}v\right\} $.
Let $r$ be an integer between $0$ and $T_{0}-1$ (or $T-1$). For
homogeneous $u\in V^{r\ast}$, set
\[
\delta_{r}=\begin{cases}
1 & \text{if}\ r=0\\
0 & \text{if}\ r\not=0.
\end{cases}
\]
For $v\in V,$ define
\[
u\circ_{g}v=\text{Res}_{z}\frac{\left(1+z\right)^{\text{wt}u-1+\delta_{r}+\frac{r}{T}}}{z^{1+\delta_{r}}}Y\left(u,z\right)v.
\]
Set
\[
O_{g}\left(V\right)=\text{span}\left\{ u\circ_{g}v\mid u,v\in V\right\} .
\]
Define the linear space $A_{g}\left(V\right)$ to be the quotient
$V/O_{g}\left(V\right).$ For short we will denote the image of $v$
in $A_{g}\left(V\right)$ by $[v].$

\begin{remark} (1) When $g=1,$ the $u\circ_{g}v$, $O_{g}\left(V\right)$
and $A_{g}\left(V\right)$ are just $u\circ v,$ $O\left(V\right)$
and $A\left(V\right)$ defined in \cite{KW}. We will use the same
notations in this case. If $V$ is a vertex operator algebra, $A_{g}\left(V\right)$
was constructed in \cite{DLM1}.

(2) When $g=\sigma$, the only eigenspace of $g\sigma$ is $V$ itself
with eigenvalue 1. For $u,v\in V$, $u\circ_{\sigma}v=\text{Res}_{z}\frac{\left(1+z\right)^{\text{wt}u}}{z^{2}}Y\left(u,z\right)v.$
We will use
\[
A_{\sigma}\left(V\right)=V/O_{\sigma}\left(V\right),
\]
where $O_{\sigma}\left(V\right)=\text{span}\left\{ u\circ_{\sigma}v\mid u,v\in V\right\} $.

\end{remark}

Now we define a product $\ast_{g}$ which induces an associative product
in $A_{g}\left(V\right)$. For homogeneous $u\in V^{r\ast}$, $v\in V$
and $0\le r\le T-1$. Set
\[
u\ast_{g}v=\begin{cases}
\text{Res}_{z}\left(Y\left(u,z\right)\frac{\left(1+z\right)^{\text{wt}u}}{z}v\right) & \text{if}\ r=0\\
0 & \text{if}\ r>0.
\end{cases}
\]
Extend $\ast_{g}$ linearly to obtain a bilinear product on $V$.
Then the restriction of $\ast_{g}$ to $V^{0\ast}$ coincides with
that of \cite{KW}. If $g=1$, we use $\ast$ instead of $\ast_{g}$.
The following results are obtained in \cite{DZ}:

\begin{theorem} \label{A_g(V) main results} (1) If $r\not=0$, then
$V^{r\ast}\subseteq O_{g}\left(V\right).$

(2) $A_{g}\left(V\right)$ is an associative algebra under $\ast_{g}$
with identity $\left[\textbf{1}\right]$ and central element $\left[\omega\right]$
where $[u]$ is the image of $u\in V$ in $A_{g}(V).$

(3) There is a bijection between simple $A_{g}\left(V\right)$-modules
and irreducible admissible $g$-twisted $V$-modules.

\end{theorem}

%The following theorem is proved in \cite{DZ}: \begin{proposition}
%There is one-to-one correspondence between the simple $A_{g}\left(V\right)$-modules
%and simple admissible $g$-twisted $V$-modules.

%\end{proposition}

\section{The isomorphisms between associative algebras}

In this section, we first recall some general notions about permutation
orbifold. In Subsection 4.1 we construct an explicit isomorphism between
$A_{g}\left(V^{\otimes k}\right)$ and $A$$\left(V\right)$ with
$k$ an odd number. In Subsection 4.2, for even number $k$ we give
the isomorphism between $A_{g}\left(V^{\otimes k}\right)$ and $A_{\sigma}\left(V\right)$.

Let $V=\left(V,Y,\textbf{1},\text{\ensuremath{\omega}}\right)$ be
a vertex operator superalgebra and $k$ be a fixed positive integer.
Then by Remark \ref{tensor product VOSA}, $V^{\otimes k}$ is also
a vertex operator superalgebra, and the permutation group $S_{k}$
acts naturally on $V^{\otimes k}.$ For instance, let $g=\left(1\ 2\cdots k\right).$
Then
\begin{alignat}{1}
g: & V\otimes V\otimes\cdots\otimes V\to V\otimes V\otimes\cdots\otimes V\nonumber \\
 & v_{1}\otimes v_{2}\otimes\cdots\otimes v_{k}\mapsto\left(-1\right)^{\left|v_{1}\right|\left(\left|v_{2}\right|+\cdots+\left|v_{k}\right|\right)}v_{2}\otimes v_{3}\otimes\cdots\otimes v_{k}\otimes v_{1}.\label{g action}
\end{alignat}

Let $n\in\left\{ 1,2,\cdots,k\right\} $ and $a_{1},\cdots,a_{n}\in\left\{ 1,2,\cdots,k\right\} $
be $n$ different numbers. For $v_{1},\cdots,v_{n}\in V$, in the
rest of this paper we denote the vector in $V^{\otimes k}$ whose
$a_{j}$-th tensor factor is $v_{j}$ and whose other tensor factors
are ${\bf 1}$ by $x_{v_{1},\cdots,v_{n}}^{a_{1},\cdots,a_{n}}$.
We call a linear combination of vectors of the form $x_{v_{1},\cdots,v_{n}}^{a_{1},\cdots,a_{n}}$
a \textit{$n$-tensor vector}. In particular, we will denote $x_{u}^{a}={\bf 1}\otimes\cdots\otimes{\bf 1}\otimes u\otimes{\bf 1}\otimes\cdots\otimes{\bf 1}$
by $u^{a}$ which is a $1$-tensor vector whose $a$-th tensor factor
is $u$ and whose other tensor factors are ${\bf 1}$.

We set $\tau_{1,1}=1$ and for $1\le i\le k$, $2\le n\le k$, set

\begin{equation}
\tau_{i,n}:=\begin{cases}
\left(-1\right)^{\left|v_{i}\right|\left(\sum_{s\in\{1,\cdots,n\}\setminus\{i\}}\left|v_{s}\right|\right)} & \text{if}\ i\le n\\
1, & \text{if}\ i>n.
\end{cases}\label{sign function}
\end{equation}
Then (\ref{g action}) implies that for any $1\le n\le k$, $1\le a\le k$,
\[
g^{a}x_{v_{1},\cdots,v_{n}}^{1,\cdots,n}=\prod_{i=1}^{a}\tau_{i,n}x_{v_{1},\cdots,v_{n}}^{1-a+k,\cdots,n-a+k},
\]
where $m-a+k$ is understood to be $m-a$ if $m-a>0$. In particular,
$g^{a}\left(u^{1}\right)=u^{1-a+k}$, $1\le a\le k$.

Following \cite{BDM}, an operator $\Delta_{k}\left(z\right)=\Delta_{k}^{V}\left(z\right)\in\left(\text{End}V\right)\left[\left[z^{1/2k},z^{-1/2k}\right]\right]$
on $V$ was defined in \cite{Ba,BW} by
\[
\Delta_{k}\left(z\right)=\exp\left(\sum_{j\in\mathbb{Z}_{+}}a_{j}z^{\frac{j}{k}}L\left(j\right)\right)\left(k^{\frac{1}{2}}\right)^{-2L\left(0\right)}\left(z^{\frac{1}{2k}\left(k-1\right)}\right)^{-2L\left(0\right)}
\]
where $a_{j}\in\mathbb{C}$ for $j\in\mathbb{Z}_{+}$ are uniquely
determined by
\[
\exp\left(-\sum_{j\in\mathbb{Z}_{+}}a_{j}x^{j+1}\frac{\partial}{\partial x}\right)\cdot x=\frac{1}{k}\left(1+x\right)^{k}-\frac{1}{k}.
\]
The following property of $\Delta_{k}\left(z\right)$ is proved in
\cite{Ba}.

\begin{proposition} In $\left(\text{End}V\right)\left[\left[z^{1/2k},z^{-1/2k}\right]\right]$,
we have
\[
\Delta_{k}^{V}\left(z\right)Y\left(u,z_{0}\right)\Delta_{k}^{V}\left(z\right)^{-1}=Y\left(\Delta_{k}^{V}\left(z+z_{0}\right)u,\left(z+z_{0}\right)^{1/k}-z^{1/k}\right),
\]
for all $u\in V$.

\end{proposition} In particular, we have
\[
\Delta_{k}\left(1\right)=\exp\left(\sum_{j\in\mathbb{Z}_{+}}a_{j}L\left(j\right)\right)k^{-L\left(0\right)},
\]
which will play an important role in the construction of the isomorphism
between $A_{g}\left(V^{\otimes k}\right)$ and $A\left(V\right)$
(resp. $A_{\sigma}\left(V\right)$) for odd number $k$ (resp. even
number $k$).

\subsection{\label{odd k isomorphism}The isomorphism between $A_{g}\left(V^{\otimes k}\right)$
and $A$$\left(V\right)$ with $k$ odd }

In this subsection, we construct an isomorphism between $A_{g}\left(V^{\otimes k}\right)$
and $A\left(V\right)$ as associative algebras where $k$ is an odd
number. We will first find out the generating set of $A_{g}\left(V^{\otimes k}\right),$
then construct an isomorphism from $A_{g}\left(V^{\otimes k}\right)$
to $A\left(V\right)$.

\begin{lemma} \label{n-tensor to n-1-tensor-odd} Let $g=\left(1\ 2\cdots k\right)$
with $k$ odd. Then $A_{g}\left(V^{\otimes k}\right)$ is spanned
by $[\sum_{a=1}^{k}v^{a}]$ for $v\in V_{\bar{0}}$.

\end{lemma}
\begin{proof}
Let $h=g\sigma.$ Then $o\left(g\sigma\right)=2k$. By Theorem \ref{A_g(V) main results},
$A_{g}\left(V^{\otimes k}\right)$ is spanned by vectors of form $\sum_{a=1}^{2k}h^{a}x_{v_{1},\cdots,v_{n}}^{a_{1},\cdots,a_{n}}$
modulo $O_{g}(V)$ for $v_{j}\in V$, $j=1,\cdots,n$ and $n=1,\cdots,k$.
Note that $\sigma\left(x_{v_{1},\cdots,v_{n}}^{a_{1},\cdots,a_{n}}\right)=(-1)^{i}x_{v_{1},\cdots,v_{n}}^{a_{1},\cdots,a_{n}}$
for $x_{v_{1},\cdots,v_{n}}^{a_{1},\cdots,a_{n}}\in\left(V^{\otimes k}\right)_{\overline{i}}$
with $i=0,1.$ It is easy to check that
\begin{alignat*}{1}
\sum_{a=1}^{2k}h^{a}x_{v_{1},\cdots,v_{n}}^{a_{1},\cdots,a_{n}} & =\begin{cases}
2\sum_{a=1}^{k}g^{a}x_{v_{1},\cdots,v_{n}}^{a_{1},\cdots,a_{n}} & \text{for\ }x_{v_{1},\cdots,v_{n}}^{a_{1},\cdots,a_{n}}\in\left(V^{\otimes k}\right)_{\bar{0}}\\
0 & \text{for}\ x_{v_{1},\cdots,v_{n}}^{a_{1},\cdots,a_{n}}\in\left(V^{\otimes k}\right)_{\bar{1}}.
\end{cases}
\end{alignat*}
Therefore, $A_{g}\left(V^{\otimes k}\right)$ is spanned by vectors
of form $\sum_{a=1}^{k}g^{a}x_{v_{1},\cdots,v_{n}}^{a_{1},\cdots,a_{n}}$
modulo $O_{g}(V)$ with $x_{v_{1},\cdots,v_{n}}^{a_{1},\cdots,a_{n}}\in\left(V^{\otimes k}\right)_{\bar{0}}$.

Let $n=2,\cdots,k$. We will prove that any vector of the form $\sum_{a=1}^{k}g^{a}\left(x_{v_{1},\cdots,v_{n}}^{a_{1},\cdots,a_{n}}\right)$
with $x_{v_{1},\cdots,v_{n}}^{a_{1},\cdots,a_{n}}\in\left(V^{\otimes k}\right){}_{\bar{0}}$
can be reduced to a $\left(n-1\right)$-tensor vector modulo $O_{g}(V)$
and hence by induction $A_{g}\left(V^{\otimes k}\right)$ is spanned
by vectors of the form\textcolor{red}{{} }$\sum_{a=1}^{k}v^{a}$,
$v\in V_{\bar{0}}$. We only give a proof for $(a_{1},\cdots,a_{n})=(1,\cdots,n)$
and the proof for general case is similar. Let $\xi=e^{\frac{\pi i}{k}}$.
Note that the eigenspace of $h$ with eigenvalue $\xi^{s}$ is given
by
\[
\left(V^{\otimes k}\right)^{s\ast}=\left\{ \sum_{a=1}^{2k}\xi^{-\left(a-1\right)s}h^{a}x_{v_{1},\cdots,v_{k}}^{1,\cdots,k}\mid v_{1},\cdots,v_{k}\in V\right\} ,
\]
where $0\le s\le2k-1$. We will use the fact that for any $\sum_{a=1}^{2k}\xi^{-\left(a-1\right)s}h^{a-1}u^{1}\in\left(V^{\otimes k}\right)^{s\ast}$
and $\sum_{b=1}^{2k}\xi^{\left(b-1\right)s}h^{b-1}x_{v_{1},\cdots,v_{n-1}}^{1,\cdots,n-1}\in\left(V^{\otimes k}\right)^{\left(2k-s\right)\ast}$
with $s=1,\cdots,2k-1,$
\[
0\equiv\left(\sum_{a=1}^{2k}\xi^{-\left(a-1\right)s}h^{a-1}u^{1}\right)\circ_{g}\left(\sum_{b=1}^{2k}\xi^{\left(b-1\right)s}h^{b-1}x_{v_{1},\cdots,v_{n-1}}^{1,\cdots,n-1}\right).
\]

(1) Let $u^{1}\in\left(V^{\otimes k}\right)_{\bar{0}}$ and $x_{v_{1},\cdots,v_{n-1}}^{1,\cdots,n-1}\in\left(V^{\otimes k}\right)_{\bar{0}}$.
Then we have $\sigma\left(u^{1}\right)=u^{1}$ and $\sigma\left(x_{v_{1},\cdots,v_{n-1}}^{1,\cdots,n-1}\right)=x_{v_{1},\cdots,v_{n-1}}^{1,\cdots,n-1}.$
Therefore
\begin{alignat}{1}
\sum_{a=1}^{2k}\xi^{-\left(a-1\right)s}h^{a-1}u^{1} & =\sum_{a=1}^{2k}\xi^{-\left(a-1\right)s}g^{a-1}u^{1}\nonumber \\
 & =\sum_{a=1}^{k}\left(1+\left(-1\right)^{s}\right)\xi^{-\left(a-1\right)s}u^{2-a+k}\nonumber \\
 & =\begin{cases}
0 & \text{if\ }s\ \text{is\ odd}\\
\sum_{a=1}^{k}2\xi^{-\left(a-1\right)s}u^{2-a+k} & \text{if\ }s\ \text{is\ even,}
\end{cases}\label{1-tensor even}
\end{alignat}

and
\begin{alignat*}{1}
\sum_{b=1}^{2k}\xi^{\left(b-1\right)s}h^{b-1}x_{v_{1},\cdots,v_{n-1}}^{1,\cdots,n-1}= & \sum_{b=1}^{2k}\xi^{\left(b-1\right)s}g^{b-1}x_{v_{1},\cdots,v_{n-1}}^{1,\cdots,n-1}\\
= & \sum_{b=1}^{k}\xi^{\left(b-1\right)s}\left(1+\left(-1\right)^{s}\right)g^{b-1}x_{v_{1},\cdots,v_{n-1}}^{1,\cdots,n-1}\\
= & \begin{cases}
0 & \text{if\ }s\ \text{is\ odd}\\
2\sum_{b=1}^{k}\xi^{\left(b-1\right)s}g^{b-1}x_{v_{1},\cdots,v_{n-1}}^{1,\cdots,n-1} & \text{if\ }s\ \text{is\ even}
\end{cases}\\
= & \begin{cases}
0 & \text{if\ }s\ \text{is\ odd}\\
2\sum_{b=1}^{k}\xi^{bs}\prod_{i=1}^{b}\tau_{i,n-1}x_{v_{1},v_{2},\cdots,v_{n-1}}^{1-b+k,2-b+k,\cdots,n-1-b+k} & \text{if\ }s\ \text{is\ even.}
\end{cases}
\end{alignat*}
Now for $1\le s\le2k-1$ with $s$ even, we have
\begin{alignat*}{1}
0 & \equiv\left(\sum_{a=1}^{2k}\xi^{-\left(a-1\right)s}h^{a-1}u^{1}\right)\circ_{g}\left(\sum_{b=1}^{2k}\xi^{\left(b-1\right)s}h^{b-1}x_{v_{1},\cdots,v_{n-1}}^{1,\cdots,n-1}\right)\\
 & \equiv\left(2\sum_{a=1}^{k}\xi^{-\left(a-1\right)s}u^{2-a+k}\right)\circ_{g}\left(2\sum_{b=1}^{k}\xi^{bs}\prod_{i=1}^{b}\tau_{i,n-1}x_{v_{1},v_{2},\cdots,v_{n-1}}^{1-b+k,2-b+k,\cdots,n-1-b+k}\right)\\
 & =4\text{Res}_{z}Y\left(\sum_{a=1}^{k}\xi^{-\left(a-1\right)s}u^{2-a+k},z\right)\left(\sum_{b=1}^{k}\xi^{bs}\prod_{i=1}^{b}\tau_{i,n-1}x_{v_{1},v_{2},\cdots,v_{n-1}}^{1-b+k,2-b+k,\cdots,n-1-b+k}\right)\frac{\left(1+z\right)^{\text{wt}u-1+\frac{s}{k}}}{z}\\
 & =4\sum_{a=1}^{k}\sum_{b=1}^{k}\prod_{i=1}^{b}\tau_{i,n-1}\xi^{\left(b-a+1\right)s}\text{Res}_{z}Y\left(u^{2-a+k},z\right)\left(x_{v_{1},v_{2},\cdots,v_{n-1}}^{1-b+k,2-b+k,\cdots,n-1-b+k}\right)\frac{\left(1+z\right)^{\text{wt}u-1+\frac{s}{k}}}{z}\\
 & =4\sum_{b=1}^{k}\ \sum_{2-a\notin\left\{ 1-b,2-b,\cdots,n-1-b\right\} }\prod_{i=1}^{b}\tau_{i,n-1}\xi^{\left(b-a+1\right)s}x_{u,v_{1},\cdots,v_{n-1}}^{2-a+k,1-b+k,2-b+k,\cdots,n-1-b+k}\\
 & +4\sum_{b=1}^{k}\prod_{i=1}^{b}\tau_{i,n-1}\sum_{j=1}^{n-1}\xi^{\left(j-1\right)s}x_{v_{1},\cdots,v_{j-1},v_{j}',v_{j+1},\cdots,v_{n-1}}^{1-b+k,\cdots,j-1-b+k,j-b+k,j+1-b+k,\cdots,n-1-b+k},
\end{alignat*}
where $v_{j}'=\text{Res}_{z}Y\left(u,z\right)v_{j}\frac{\left(1+z\right)^{\text{wt}u-1+\frac{s}{k}}}{z}.$
For $1\le t\le k-n+1,1\le s\le k-1$, set
\[
y_{t}=\sum_{b=1}^{k}g^{b-1}x_{u,v_{1},\cdots,v_{n-1}}^{n-1+t,1,\cdots,n-1}
\]
and
\[
u_{s}=-\sum_{b=1}^{k}\prod_{i=1}^{b}\tau_{i,n-1}\sum_{j=1}^{n-1}\xi^{\left(j-1\right)s}x_{v_{1},\cdots,v_{j-1},v_{j}',v_{j+1},\cdots,v_{n-1}}^{1-b+k,\cdots,j-1-b+k,j-b+k,j+1-b+k,\cdots,n-1-b+k}.
\]
Then
\begin{align*}
 & \sum_{b=1}^{k}\sum_{2-a\notin\left\{ 1-b,2-b,\cdots,n-1-b\right\} }\prod_{i=1}^{b}\tau_{i,n-1}\xi^{\left(b-a+1\right)s}x_{u,v_{1},\cdots,v_{n-1}}^{2-a+k,1-b+k,2-b+k,\cdots,n-1-b+k}\\
= & \sum_{j=1}^{k-n+1}\xi^{(n+j-2)s}y_{j}\equiv u_{s}.
\end{align*}
Since $s$ is even, the system $\sum_{j=1}^{k-n+1}\xi^{(n+j-2)s}y_{j}\equiv u_{s},$
$1\le s\le2k-1$ is equivalent to $\sum_{j=1}^{k-n+1}\eta^{(n+j-2)t}y_{j}\equiv u_{t},$
$1\le t\le k-1$ where $\eta=\xi^{2}=e^{\frac{2\pi i}{k}}.$ It is
clear that the linear system has a solution such that each $n$-tensor
vector $y_{t}\in\left(V^{\otimes k}\right)_{\bar{0}}^{0\ast}$ is
a linear combination of $\left(n-1\right)$-tensor vector $u_{s}\in\left(V^{\otimes k}\right)_{\bar{0}}$
for all $t$. In particular, when $n=2$,
\begin{gather}
y_{j}=\sum_{b=1}^{k}g^{b-1}x_{u,v_{1}}^{1+j,1},j=1,\cdots,k-1;u,v_{1}\in V_{\bar{0}};\nonumber \\
u_{t}=-\sum_{b=1}^{k}x_{v'_{1}}^{1-b+k}=-\sum_{a=1}^{k}x_{v'_{1}}^{a},t=1,2,\cdots,k-1;v'_{1}\in V_{\bar{0}}.\label{odd u_t expresssion}
\end{gather}

(2) Let $u^{1}\in\left(V^{\otimes k}\right)_{\bar{1}}$ and $x_{v_{1},\cdots,v_{n-1}}^{1,\cdots,n-1}\in\left(V^{\otimes k}\right)_{\bar{1}}.$
Then $\sigma\left(u^{1}\right)=-u^{1}$ and $\sigma\left(x_{v_{1},\cdots,v_{n-1}}^{1,\cdots,n-1}\right)=-x_{v_{1},\cdots,v_{n-1}}^{1,\cdots,n-1}.$
Then by similar arguments, we obtain
\[
\sum_{a=1}^{2k}\xi^{-\left(a-1\right)s}h^{a-1}u^{1}=\begin{cases}
0 & \text{if\ }s\ \text{is\ even}\\
2\sum_{a=1}^{k}\xi^{-\left(a-1\right)s}\left(-1\right)^{a-1}u^{k-a+2} & \text{if\ }s\ \text{is\ odd}
\end{cases}
\]
and
\begin{equation}
\sum_{b=1}^{2k}\xi^{\left(b-1\right)s}h^{b-1}x_{v_{1},\cdots,v_{n-1}}^{1,\cdots,n-1}=\begin{cases}
0 & \text{if\ }s\ \text{is\ even}\\
2\sum_{b=1}^{k}\xi^{bs}\left(-1\right)^{b}\prod_{i=1}^{b}\tau_{i,n-1}x_{v_{1},\cdots,v_{n-1}}^{1-b+k,2-b+k,\cdots,n-1-b+k} & \text{if\ }s\ \text{is\ odd.}
\end{cases}\label{n-1 tensor s odd nonzero}
\end{equation}
Now for $1\le s\le2(k-n+1)-1$ with $s$ odd, we have
\begin{alignat*}{1}
0\equiv & \left(\sum_{a=1}^{2k}\xi^{-\left(a-1\right)s}h^{a-1}u^{1}\right)\text{\ensuremath{\circ}}_{g}\left(\sum_{b=1}^{2k}\xi^{\left(b-1\right)s}h^{b-1}x_{v_{1},\cdots,v_{n-1}}^{1,\cdots,n-1}\right)\\
= & \text{Res}_{z}Y\left(2\sum_{a=1}^{k}\xi^{-\left(a-1\right)s}\left(-1\right)^{a-1}u^{k-a+2},z\right)\\
 & \cdot\left(2\sum_{b=1}^{k}\xi^{bs}\left(-1\right)^{b}\prod_{i=1}^{b}\tau_{i,n-1}x_{v_{1},\cdots,v_{n-1}}^{1-b+k,2-b+k,\cdots,n-1-b+k}\right)\frac{\left(1+z\right)^{\text{wt}u-1+\frac{s}{k}}}{z}\\
 & =4\sum_{a=1}^{k}\sum_{b=1}^{k}\left(-1\right)^{a+b-1}\prod_{i=1}^{b-1}\tau_{i,n-1}\xi^{\left(b-a+1\right)s}\text{Res}_{z}\left(Y\left(u^{k-a+2},z\right)x_{v_{1},\cdots,v_{n-1}}^{1-b+k,2-b+k,\cdots,n-1-b+k}\right)\frac{\left(1+z\right)^{\text{wt}u-1+\frac{s}{k}}}{z}\\
 & =4\sum_{b=1}^{k}\prod_{i=1}^{b-1}\tau_{i,n-1}\sum_{2-a\notin\left\{ 1-b,2-b,\cdots,n-1-b\right\} }\left(-1\right)^{a+b-1}\xi^{\left(b-a+1\right)s}\epsilon_{a}x_{u,v_{1},\cdots,v_{n-1}}^{2-a+k,1-b+k,2-b+k,\cdots,n-1-b+k}\\
 & +4\sum_{b=1}^{k}\prod_{i=1}^{b}\tau_{i,n-1}\sum_{j=1}^{n-1}\left(-1\right)^{j-1}\xi^{\left(j-1\right)s}\epsilon_{j}x_{v_{1},\cdots,v_{j-1},v_{j}',v_{j+1},\cdots,v_{n-1}}^{1-b+k,\cdots,j-1-b+k,j-b+k,j+1-b+k,\cdots,n-1-b+k},
\end{alignat*}
where $v_{j}'=\text{Res}_{z}Y\left(u,z\right)v_{j}\frac{\left(1+z\right)^{\text{wt}u-1+\frac{s}{k}}}{z}\in V_{\bar{0}},$
and $\epsilon_{a},\epsilon_{j}=\pm1$ are determined by (\ref{tensor product-super sign}).
For $1\le t\le k-n+1$ and $1\le s\le2(k-n+1)-1$ with $s$ odd, set
\begin{gather*}
y_{t}=\left(-1\right)^{n+t}\epsilon_{t}\sum_{b=1}^{k}g^{b-1}x_{u,v_{1},\cdots,v_{n-1}}^{n-1+t,1,\cdots,n-1},\\
u_{s}=-\sum_{b=1}^{k}\prod_{i=1}^{b}\tau_{i,n-1}\sum_{j=1}^{n-1}\left(-1\right)^{j-1}\xi^{\left(j-1\right)s}\epsilon_{j}x_{v_{1},\cdots,v_{j-1},v_{j}',v_{j+1},\cdots,v_{n-1}}^{1-b+k,\cdots,j-1-b+k,j-b+k,j+1-b+k,\cdots,n-1-b+k}.
\end{gather*}
Then $y_{t}\in\left(V^{\otimes k}\right)_{\bar{0}}^{0\ast}$ and $u_{s}\in\left(V^{\otimes k}\right)_{\bar{0}}$
and now we obtain
\begin{alignat*}{1}
 & \sum_{b=1}^{k}\sum_{2-a\notin\left\{ 1-b,2-b,\cdots,n-1-b\right\} }\left(-1\right)^{a+b-1}\prod_{i=1}^{b}\tau_{i,n-1}\xi^{\left(b-a+1\right)s}\epsilon_{a}x_{u,v_{1},\cdots,v_{n-1}}^{2-a+k,1-b+k,2-b+k,\cdots,n-1-b+k}\\
= & \sum_{j=1}^{k-n+1}\xi^{(n+j-2)s}y_{j}\equiv u_{s},
\end{alignat*}
where $1\le s\le2(k-n+1)-1$ with $s$ odd. It is clear that the linear
system has a solution such that each $n$-tensor vector $y_{t}$ is
a linear combination of $\left(n-1\right)$-tensor vector $u_{s}$
for all $t$. In particular, when $n=2$,
\begin{alignat*}{1}
 & y_{t}=\left(-1\right)^{t+1}\sum_{b=1}^{k}g^{b-1}x_{u,v_{1}}^{1+t,1},t=1,\cdots,k-1;u,v_{1}\in V_{\bar{1}};\\
 & u_{s}=\sum_{b=1}^{k}x_{v'_{1}}^{1-b+k}=\sum_{a=1}^{k}x_{v'_{1}}^{a},s=1,3,\cdots2k-3;v'_{1}\in V_{\bar{0}}.
\end{alignat*}
Combining (1) and (2), by induction on $n$ we see that any vector
in $\left(V^{\otimes k}\right)_{\bar{0}}^{0\ast}$ is spanned by vectors
of the form $[\sum_{a=1}^{k}v^{a}]$ for $v\in V_{\bar{0}}$.
\end{proof}
\begin{remark} \label{remark odd} For $n=2$ and $j=1,\cdots,k-1$,
we will need an explicit expression of $y_{j}$ in terms of $u_{s}$
for later purpose. Let $n=2$ in the proof of Lemma \ref{n-tensor to n-1-tensor-odd}
(1), we obtain $\sum_{j=1}^{k-1}\eta^{jt}y_{j}=u_{t},1\le t\le k-1$.
Applying \cite[Lemma 3.2]{DXY}, we obtain that $y_{j}=\sum_{s=1}^{k-1}\frac{1}{k}\left(\eta^{-js}-1\right)u_{s}$
with $u_{s}$ given by (\ref{odd u_t expresssion}).

\end{remark}

Now we give the first main theorem of this paper.

\begin{theorem}\label{mainthm-odd} Let $g=\left(1\ 2\cdots k\right)$
with $k$ odd. Define
\begin{alignat*}{1}
\phi: & \ A_{g}\left(V^{\otimes k}\right)\to A\left(V\right)\\
 & \left[\sum_{a=1}^{k}u^{a}\right]\mapsto\left[k\Delta_{k}\left(1\right)u\right],u\in V_{\bar{0}}.
\end{alignat*}
Then $\phi$ gives an isomorphism between $A_{g}\left(V^{\otimes k}\right)$
and $A\left(V\right).$

\end{theorem}
\begin{proof}
Let $h=g\sigma$ as before. Recall that $A_{g}\left(V^{\otimes k}\right)=V^{\otimes k}/O_{g}\left(V^{\otimes k}\right)$
and $A\left(V\right)=V/O\left(V\right)$ where $O_{g}\left(V^{\otimes k}\right)$
and $O\left(V\right)$ are defined in Section \ref{sec:Basics}. To
show that $\phi$ is an isomorphism between the associative algebras
$A_{g}\left(V^{\otimes k}\right)$ and $A\left(V\right)$, we first
need to show that $\phi$ is well-defined.

Since $A_{g}\left(V^{\otimes k}\right)=\left(V^{\otimes k}\right){}^{0\ast}/O_{g}\left(V^{\otimes k}\right)\cap\left(V^{\otimes k}\right){}^{0\ast}$
and $O_{g}\left(V^{\otimes k}\right)\cap\left(V^{\otimes k}\right){}^{0\ast}$
is spanned by $u\circ_{g}v$ for $u\in\left(V^{\otimes k}\right){}^{s\ast}$
and $v\in\left(V^{\otimes k}\right){}^{\left(2k-s\right)\ast}$, we
map $u\circ_{g}v$ to $0$ for $s=1,...,2k-1.$ In fact, from the
proof of Lemma \ref{n-tensor to n-1-tensor-odd} we see that $\left[u\circ_{g}v\right]=0$
just gives an identification between $p$-tensor vectors and $q$-tensor
vectors in $A_{g}\left(V^{\otimes k}\right)$ such that either $p>1$
or $q>1.$ So the main task is to show that for any $u,v\in V^{\otimes k}$,
$\phi\left(u\circ_{g}v\right)$$\in O\left(V\right).$

Let $\overline{u}=\sum_{a=1}^{k}u^{a}$, $\overline{v}=\sum_{b=1}^{k}v^{b}\in\left(V^{\otimes k}\right)^{0\ast}$
where $u,v\in V_{\bar{0}}$. Then we have
\[
\overline{u}\circ_{g}\overline{v}=\text{Res}_{z}\sum_{a,b=1}^{k}Y\left(u^{a},z\right)v^{b}\frac{\left(1+z\right)^{\text{wt}u}}{z^{2}}=\sum_{j=1}^{k}\left(u\circ v\right)^{j}+\text{wt}u\sum_{j=1}^{k-1}x_{j}+\sum_{j=1}^{k-1}y_{j},
\]
where
\[
x_{j}=\sum_{m=0}^{k-1}g^{m}x_{u,v}^{1,1+j}=\sum_{m=0}^{k-1}x_{u,v}^{1-m+k,1+j-m+k}\equiv\sum_{t=1}^{k-1}\frac{1}{k}\left(\eta^{-jt}-1\right)u_{t},
\]
\[
y_{j}=\sum_{m=0}^{k-1}g^{m}x_{u_{-2}\textbf{1},v}^{1,j}=\sum_{m=0}^{k-1}x_{u_{-2}\textbf{1},v}^{1-m+k,1+j-m+k}\equiv\sum_{t=1}^{k-1}\frac{1}{k}\left(\eta^{-jt}-1\right)w_{t}
\]
with
\[
u_{t}=-\text{Res}_{z}\sum_{j=1}^{k}\left(Y\left(u,z\right)v\frac{\left(1+z\right)^{\text{wt}u-1+\frac{t}{k}}}{z}\right)^{j},
\]
and
\[
w_{t}=-\text{Res}_{z}\sum_{j=1}^{k}\left(Y\left(u_{-2}\textbf{1},z\right)v\frac{\left(1+z\right)^{\text{wt}u+\frac{t}{k}}}{z}\right)^{j},
\]
where we use Remark \ref{remark odd}. By the same arguments given
in \cite[Theorem 3.7]{DXY}, one can prove that $\phi$ is an isomorphism
between $A_{g}\left(V^{\otimes k}\right)$ and $A\left(V\right).$
\end{proof}

\subsection{The isomorphism between $A_{g}\left(V^{\otimes k}\right)$ and $A_{\sigma}\left(V\right)$
with $k$ even}

In this subsection, we construct an isomorphism between $A_{g}\left(V^{\otimes k}\right)$
and $A_{\sigma}\left(V\right)$ as associative algebras where $k$
is an even number. Similar to the last subsection, we will first find
the generating set of $A_{g}\left(V^{\otimes k}\right)$ and then
construct an isomorphism from $A_{g}\left(V^{\otimes k}\right)$ to
$A_{\sigma}\left(V\right).$

\begin{lemma} \label{n-tensor to n-1-tensor-1-even} Let $g=\left(1\ 2\cdots k\right)$
with $k$ even. Then $A_{g}\left(V^{\otimes k}\right)$ is spanned
by $[\sum_{a=1}^{k}v^{a}]\text{ for \ensuremath{v\in V_{\bar{0}}}}$
and $[\sum_{a=1}^{k}\left(-1\right)^{a-1}v^{a}]$ for $v\in V_{\bar{1}}.$

\end{lemma}
\begin{proof}
Denote $h=g\sigma.$ Then $o\left(g\sigma\right)=k$. By Theorem \ref{A_g(V) main results},
$A_{g}\left(V^{\otimes k}\right)$ is spanned by vectors of form $\sum_{a=1}^{k}h^{a}x_{v_{1},\cdots,v_{n}}^{a_{1},\cdots,a_{n}}$
modulo $O_{g}\left(V^{\otimes k}\right)$ for $v_{j}\in V$, $j=1,\cdots,n$.
Note that $\sigma\left(x_{v_{1},\cdots,v_{n}}^{a_{1},\cdots,a_{n}}\right)=(-1)^{i}x_{v_{1},\cdots,v_{n}}^{a_{1},\cdots,a_{n}}$
for $x_{v_{1},\cdots,v_{n}}^{a_{1},\cdots,a_{n}}\in\left(V^{\otimes k}\right)_{\bar{i}}$
with $i=0,1.$ It is easy to check that
\[
\sum_{a=1}^{k}h^{a}x_{v_{1},\cdots,v_{n}}^{a_{1},\cdots,a_{n}}=\begin{cases}
\sum_{a=1}^{k}g^{a}x_{v_{1},\cdots,v_{n}}^{a_{1},\cdots,a_{n}} & \text{if\ }x_{v_{1},\cdots,v_{n}}^{a_{1},\cdots,a_{n}}\in\left(V^{\otimes k}\right)_{\bar{0}}\\
\sum_{a=1}^{k}g^{a}(-1)^{a}x_{v_{1},\cdots,v_{n}}^{a_{1},\cdots,a_{n}} & \text{if}\ x_{v_{1},\cdots,v_{n}}^{a_{1},\cdots,a_{n}}\in\left(V^{\otimes k}\right)_{\bar{1}}.
\end{cases}
\]
Let $n=2,\cdots,k$. We will first prove that any vector of the form
$\sum_{a=1}^{k}g^{a}\left(x_{v_{1},\cdots,v_{n}}^{a_{1},\cdots,a_{n}}\right)$
with $x_{v_{1},\cdots,v_{n}}^{a_{1},\cdots,a_{n}}\in\left(V^{\otimes k}\right)_{\bar{0}}$
and $\sum_{a=1}^{k}g^{a}(-1)^{a}x_{v_{1},\cdots,v_{n}}^{a_{1},\cdots,a_{n}}$
with $x_{v_{1},\cdots,v_{n}}^{a_{1},\cdots,a_{n}}\in\left(V^{\otimes k}\right)_{\bar{1}}$
can be reduced to a $\left(n-1\right)$-tensor vector modulo $O_{g}\left(V^{\otimes k}\right)$
and then by induction we will prove that $A_{g}\left(V^{\otimes k}\right)$
is spanned by vectors of the form $\sum_{a=1}^{k}v^{a}$ for $v\in V_{\bar{0}}$
and $\left[\sum_{a=1}^{k}\left(-1\right)^{a-1}v^{a}\right]$ for $v\in V_{\bar{1}}.$
We only give a proof for $(a_{1},\cdots,a_{n})=(1,\cdots,n)$ and
the proof for general case is similar. Note that the eigenspace of
$h$ with eigenvalue $\eta^{s}$ is given by $\left(V^{\otimes k}\right)^{s\ast}=\left\{ \sum_{a=1}^{k}\eta^{-\left(a-1\right)s}h^{a}x_{v_{1},\cdots,v_{k}}^{1,\cdots,k}\mid v_{1},\cdots,v_{k}\in V\right\} $,
where $\eta=e^{\frac{2\pi i}{k}}$ and $0\le s\le k$. We will use
the fact that for $\sum_{a=1}^{k}\eta^{-\left(a-1\right)s}h^{a-1}u^{1}\in\left(V^{\otimes k}\right)^{s\ast}$
and $\sum_{b=1}^{k}\eta^{\left(b-1\right)s}h^{b-1}x_{v_{1},\cdots,v_{n-1}}^{1,\cdots,n-1}\in\left(V^{\otimes k}\right)^{\left(k-s\right)\ast}$
with $s=1,\cdots,k-1,$
\[
0\equiv\left(\sum_{a=1}^{k}\eta^{-\left(a-1\right)s}h^{a-1}u^{1}\right)\circ_{g}\left(\sum_{b=1}^{k}\eta^{\left(b-1\right)s}h^{b-1}x_{v_{1},\cdots,v_{n-1}}^{1,\cdots,n-1}\right).
\]

(1) First we prove that any $n$-tensor vector in $\left(V^{\otimes k}\right)_{\bar{0}}^{0\ast}$
can be reduced to vectors of the form $\left[\sum_{a=1}^{k}v^{a}\right]$
for $v\in V_{\bar{0}}$.

(i) Let $u^{1}\in V_{\bar{0}}$ and $x_{v_{1},\cdots,v_{n-1}}^{1,\cdots,n-1}\in\left(V^{\otimes k}\right)_{\bar{0}}.$
Then $\sigma\left(u^{1}\right)=u^{1}$ and $\sigma\left(x_{v_{1},\cdots,v_{n-1}}^{1,\cdots,n-1}\right)=x_{v_{1},\cdots,v_{n-1}}^{1,\cdots,n-1}.$
Therefore we have
\[
\sum_{a=1}^{k}\eta^{-\left(a-1\right)s}h^{a-1}u^{1}=\sum_{a=1}^{k}\eta^{-\left(a-1\right)s}g^{a-1}u^{1}=\sum_{a=1}^{k}\eta^{-\left(a-1\right)s}u^{2-a+k}
\]
and
\begin{alignat*}{1}
\sum_{b=1}^{k}\eta^{\left(b-1\right)s}h^{b-1}x_{v_{1},\cdots,v_{n-1}}^{1,\cdots,n-1} & =\sum_{b=1}^{k}\eta^{\left(b-1\right)s}g^{b-1}x_{v_{1},\cdots,v_{n-1}}^{1,\cdots,n-1}\\
 & =\sum_{b=1}^{k}\eta^{bs}g^{b}x_{v_{1},\cdots,v_{n-1}}^{1,\cdots,n-1}\\
 & =\sum_{b=1}^{k}\eta^{bs}\prod_{i=1}^{b}\tau_{i,n-1}x_{v_{1},v_{2},\cdots,v_{n-1}}^{1-b+k,2-b+k,\cdots,n-1-b+k},
\end{alignat*}
where $m-b+k$ is understood to be $m-b$ if $m-b>0$ and $\tau_{i,n-1}$
is defined in (\ref{sign function}). Now we obtain
\begin{alignat*}{1}
0 & \equiv\left(\sum_{a=1}^{k}\eta^{-\left(a-1\right)s}h^{a-1}u^{1}\right)\circ_{g}\left(\sum_{b=1}^{k}\eta^{\left(b-1\right)s}h^{b-1}x_{v_{1},\cdots,v_{n-1}}^{1,\cdots,n-1}\right)\\
 & \equiv\left(\sum_{a=1}^{k}\eta^{-\left(a-1\right)s}u^{2-a+k}\right)\circ_{g}\left(\sum_{b=1}^{k}\eta^{bs}g^{b}x_{v_{1},\cdots,v_{n-1}}^{1,\cdots,n-1}\right)\\
 & =\text{Res}_{z}Y\left(\sum_{a=1}^{k}\eta^{-\left(a-1\right)s}u^{2-a+k},z\right)\left(\sum_{b=1}^{k}\eta^{bs}\prod_{i=1}^{b}\tau_{i,n-1}x_{v_{1},v_{2},\cdots,v_{n-1}}^{1-b+k,2-b+k,\cdots,n-1-b+k}\right)\frac{\left(1+z\right)^{\text{wt}u-1+\frac{s}{k}}}{z}\\
 & =\sum_{a=1}^{k}\sum_{b=1}^{k}\prod_{i=1}^{b}\tau_{i,n-1}\eta^{\left(b-a+1\right)s}\text{Res}_{z}Y\left(u^{2-a+k},z\right)\left(x_{v_{1},v_{2},\cdots,v_{n-1}}^{1-b+k,2-b+k,\cdots,n-1-b+k}\right)\frac{\left(1+z\right)^{\text{wt}u-1+\frac{s}{k}}}{z}\\
 & =\sum_{b=1}^{k}\sum_{2-a\notin\left\{ 1-b,2-b,\cdots,n-1-b\right\} }\prod_{i=1}^{b}\tau_{i,n-1}\eta^{\left(b-a+1\right)s}x_{u,v_{1},\cdots,v_{n-1}}^{2-a+k,1-b+k,2-b+k,\cdots,n-1-b+k}\\
 & +\sum_{b=1}^{k}\prod_{i=1}^{b}\tau_{i,n-1}\sum_{j=1}^{n-1}\eta^{\left(j-1\right)s}x_{v_{1},\cdots,v_{j-1},v_{j}',v_{j+1},\cdots,v_{n-1}}^{1-b+k,\cdots,j-1-b+k,j-b+k,j+1-b+k,\cdots,n-1-b+k},
\end{alignat*}
where $v_{j}'=\text{Res}_{z}Y\left(u,z\right)v_{j}\frac{\left(1+z\right)^{\text{wt}u-1+\frac{s}{k}}}{z}.$
For $1\le t\le k-n+1,$ $1\le s\le k-1$, set
\[
y_{t}=\sum_{b=1}^{k}g^{b-1}x_{u,v_{1},\cdots,v_{n-1}}^{n-1+t,1,\cdots,n-1}
\]
and
\[
u_{s}=-\sum_{b=1}^{k}\prod_{i=1}^{b}\tau_{i,n-1}\eta^{\left(j-1\right)s}x_{v_{1},\cdots,v_{j-1},v_{j}',v_{j+1},\cdots,v_{n-1}}^{1-b+k,\cdots,j-1-b+k,j-b+k,j+1-b+k,\cdots,n-1-b+k}.
\]
Then
\begin{align*}
 & \sum_{b=1}^{k}\sum_{2-a\notin\left\{ 1-b,2-b,\cdots,n-1-b\right\} }\prod_{i=1}^{b}\tau_{i,n-1}\eta^{\left(b-a+1\right)s}x_{u,v_{1},\cdots,v_{n-1}}^{2-a+k,1-b+k,2-b+k,\cdots,n-1-b+k}\\
= & \sum_{j=1}^{k-n+1}\eta^{(n+j-2)s}y_{j}\equiv u_{s}.
\end{align*}
It is clear that the linear system has a solution such that each $y_{t}$
is a linear combination of $u_{s}$ for all $t$. In particular, when
$n=2$,
\begin{gather}
y_{t}=\sum_{b=1}^{k}g^{b-1}x_{u,v_{1}}^{1+t,1}=\sum_{a=1}^{k}x_{u,v_{1}}^{a+t,a},t=1,\cdots,k;u,v_{1}\in V_{\bar{0}};\nonumber \\
u_{s}=-\sum_{b=1}^{k}x_{v'_{1}}^{1-b+k}=-\sum_{a=1}^{k}x_{v'_{1}}^{a},v'_{1}\in V_{\bar{0}}.\label{U_s even even}
\end{gather}

(ii) Let $u^{1}\in\left(V^{\otimes k}\right)_{\bar{1}}$ and $x_{v_{1},\cdots,v_{n-1}}^{1,\cdots,n-1}\in\left(V^{\otimes k}\right)_{\bar{1}}.$
Then $\sigma\left(u^{1}\right)=-u^{1}$ and $\sigma\left(x_{v_{1},\cdots,v_{n-1}}^{1,\cdots,n-1}\right)=-x_{v_{1},\cdots,v_{n-1}}^{1,\cdots,n-1}$.
Thus we have

\[
\sum_{a=1}^{k}\eta^{-\left(a-1\right)s}h^{a-1}u^{1}=\sum_{a=1}^{k}\eta^{-\left(a-1\right)s}\left(-1\right)^{a-1}u^{k-a+2}
\]
and
\[
\sum_{b=1}^{k}\eta^{\left(b-1\right)s}h^{b-1}x_{v_{1},\cdots,v_{n-1}}^{1,\cdots,n-1}=\sum_{b=1}^{k}\eta^{bs}\left(-1\right)^{b}\prod_{i=1}^{b}\tau_{i,n-1}x_{v_{1},\cdots,v_{n-1}}^{1-b+k,2-b+k,\cdots,n-1-b+k}.
\]
Now we get

\begin{alignat*}{1}
0\equiv & \left(\sum_{a=1}^{k}\eta^{-\left(a-1\right)s}h^{a-1}u^{1}\right)\text{\ensuremath{\circ}}_{g}\left(\sum_{b=1}^{k}\eta^{\left(b-1\right)s}h^{b-1}x_{v_{1},\cdots,v_{n-1}}^{1,\cdots,n-1}\right)\\
= & \text{Res}_{z}Y\left(\sum_{a=1}^{k}\eta^{-\left(a-1\right)s}\left(-1\right)^{a-1}u^{k-a+2},z\right)\\
 & \cdot\left(\sum_{b=1}^{k}\eta^{bs}\left(-1\right)^{b}\prod_{i=1}^{b}\tau_{i,n-1}x_{v_{1},\cdots,v_{n-1}}^{1-b+k,2-b+k,\cdots,n-1-b+k}\right)\frac{\left(1+z\right)^{\text{wt}u-1+\frac{s}{k}}}{z}\\
= & \sum_{a=1}^{k}\sum_{b=1}^{k}\left(-1\right)^{a+b-1}\prod_{i=1}^{b}\tau_{i,n-1}\eta^{\left(b-a+1\right)s}\text{Res}_{z}\left(Y\left(u^{k-a+2},z\right)x_{v_{1},\cdots,v_{n-1}}^{1-b+k,2-b+k,\cdots,n-1-b+k}\right)\frac{\left(1+z\right)^{\text{wt}u-1+\frac{s}{k}}}{z}\\
= & \sum_{b=1}^{k}\sum_{2-a\notin\left\{ 1-b,2-b,\cdots,n-1-b\right\} }\left(-1\right)^{a+b-1}\epsilon_{a}\prod_{i=1}^{b}\tau_{i,n-1}\eta^{\left(b-a+1\right)s}x_{u,v_{1},\cdots,v_{n-1}}^{2-a+k,1-b+k,2-b+k,\cdots,n-1-b+k}\\
 & +\sum_{b=1}^{k}\prod_{i=1}^{b}\tau_{i,n-1}\sum_{j=1}^{n-1}\left(-1\right)^{j-1}\eta^{\left(j-1\right)s}\epsilon_{j}x_{v_{1},\cdots,v_{j-1},v_{j}',v_{j+1},\cdots,v_{n-1}}^{1-b+k,\cdots,j-1-b+k,j-b+k,j+1-b+k,\cdots,n-1-b+k},
\end{alignat*}
where $v_{j}'=\text{Res}_{z}Y\left(u,z\right)v_{j}\frac{\left(1+z\right)^{\text{wt}u-1+\frac{s}{k}}}{z}$,
and $\epsilon_{a},\epsilon_{j}=\pm1$ are determined by (\ref{tensor product-super sign}).
For $1\le t\le k-n+1,$ $1\le s\le k-1$, set
\[
y_{t}=\left(-1\right)^{t+n-1}\epsilon_{t}\sum_{b=1}^{k}g^{b-1}x_{u,v_{1},\cdots,v_{n-1}}^{n-1+t,1,\cdots,n-1},
\]
\[
u_{s}=-\sum_{b=1}^{k}\prod_{i=1}^{b}\tau_{i,n-1}\sum_{j=1}^{n-1}\left(-1\right)^{1-j}\eta^{\left(j-1\right)s}\epsilon_{j}x_{v_{1},\cdots,v_{j-1},v_{j}',v_{j+1},\cdots,v_{n-1}}^{1-b+k,\cdots,j-1-b+k,j-b+k,j+1-b+k,\cdots,n-1-b+k}
\]
where $\epsilon_{t}$ is determined by (\ref{tensor product-super sign})$.$
Then
\begin{align*}
 & \sum_{b=1}^{k}\sum_{2-a\notin\left\{ 1-b,2-b,\cdots,n-1-b\right\} }\left(-1\right)^{a+b-1}\epsilon_{a}\prod_{i=1}^{b}\tau_{i,n-1}\eta^{\left(b-a+1\right)s}x_{u,v_{1},\cdots,v_{n-1}}^{2-a+k,1-b+k,2-b+k,\cdots,n-1-b+k}\\
= & \sum_{j=1}^{k-n+1}\eta^{(n+j-2)s}y_{j}\equiv u_{s}.
\end{align*}
Clearly, this linear system has a solution that each $n$-tensor vector
$y_{t}\in\left(V^{\otimes k}\right)_{\bar{0}}^{0\ast}$ is a linear
combination of $\left(n-1\right)$-tensor vector $u_{s}\in\left(V^{\otimes k}\right)_{\bar{0}}$
for all $t.$ In particular, when $n=2$,
\begin{gather}
y_{t}=\left(-1\right)^{t}\sum_{b=1}^{k}g^{b-1}x_{u,v_{1}}^{1+t,1}=\left(-1\right)^{t}\sum_{a=1}^{k}x_{u,v_{1}}^{a+t,a},t=1,\cdots,k-1;u,v\in V_{\bar{1}};\nonumber \\
u_{s}=-\sum_{b=1}^{k}x_{v'_{1}}^{2-b+k}=-\sum_{a=1}^{k}x_{v'_{1}}^{a},s=1,\cdots,k-1,v_{1}'\in V_{\bar{0}}.\label{U_s odd odd}
\end{gather}

Combining (i) and (ii), by induction on $n$ we obtain that any $n$-tensor
vector in $\left(V^{\otimes k}\right)_{\bar{0}}^{0\ast}$ can be reduced
to vectors of the form $\left[\sum_{a=1}^{k}v^{a}\right]$ for $v\in V_{\bar{0}}$.

(2) Now we prove that any $n$-tensor vector in $\left(V^{\otimes k}\right)_{\bar{1}}^{0\ast}$
can be reduced to vectors of the form $\left[\sum_{a=1}^{k}\left(-1\right)^{a-1}v^{a}\right]=\left[v^{1}-v^{2}+v^{3}-\cdots+v^{k-1}-v^{k}\right]$
for $v\in V_{\bar{1}}$.

(i) Let $u\in V_{\bar{0}}$ and $x_{v_{1},\cdots,v_{n-1}}^{1,\cdots,n-1}\in\left(V^{\otimes k}\right)_{\bar{1}}.$
Then $\sigma\left(u\right)=u$ and $\sigma\left(x_{v_{1},\cdots,v_{n-1}}^{1,\cdots,n-1}\right)=-x_{v_{1},\cdots,v_{n-1}}^{1,\cdots,n-1}.$
Therefore

\[
\sum_{a=1}^{k}\eta^{-\left(a-1\right)s}h^{a-1}u^{1}=\sum_{a=1}^{k}\eta^{-\left(a-1\right)s}g^{a-1}u^{1}=\sum_{a=1}^{k}\eta^{-\left(a-1\right)s}u^{2-a+k}
\]
and

\[
\sum_{b=1}^{k}\eta^{\left(b-1\right)s}h^{b-1}x_{v_{1},\cdots,v_{n-1}}^{1,\cdots,n-1}=\sum_{b=1}^{k}\eta^{bs}\left(-1\right)^{b}\prod_{i=1}^{b}\tau_{i,n-1}x_{v_{1},\cdots,v_{n-1}}^{1-b+k,2-b+k,\cdots,n-1-b+k}.
\]
Now we obtain
\begin{alignat*}{1}
0\equiv & \left(\sum_{a=1}^{k}\eta^{-\left(a-1\right)s}h^{a-1}u^{1}\right)\text{\ensuremath{\circ}}_{g}\left(\sum_{b=1}^{k}\eta^{\left(b-1\right)s}h^{b-1}x_{v_{1},\cdots,v_{n-1}}^{1,\cdots,n-1}\right)\\
= & \text{Res}_{z}Y\left(\sum_{a=1}^{k}\eta^{-\left(a-1\right)s}u^{k-a+2},z\right)\left(\sum_{b=1}^{k}\eta^{bs}\left(-1\right)^{b}\prod_{i=1}^{b}\tau_{i,n-1}x_{v_{1},\cdots,v_{n-1}}^{1-b+k,2-b+k,\cdots,n-1-b+k}\right)\frac{\left(1+z\right)^{\text{wt}u-1+\frac{s}{k}}}{z}\\
= & \sum_{a=1}^{k}\sum_{b=1}^{k}\left(-1\right)^{b}\prod_{i=1}^{b}\tau_{i,n-1}\eta^{\left(b-a+1\right)s}\text{Res}_{z}\left(Y\left(u^{k-a+2},z\right)x_{v_{1},v_{2},\cdots,v_{n-1}}^{1-b+k,2-b+k,\cdots,n-1-b+k}\right)\frac{\left(1+z\right)^{\text{wt}u-1+\frac{s}{k}}}{z}\\
= & \sum_{b=1}^{k}\sum_{2-a\notin\left\{ 1-b,2-b,\cdots,n-1-b\right\} }\left(-1\right)^{b}\prod_{i=1}^{b}\tau_{i,n-1}\eta^{\left(b-a+1\right)s}x_{u,v_{1},\cdots,v_{n-1}}^{2-a+k,1-b+k,2-b+k,\cdots,n-1-b+k}\\
 & +\sum_{b=1}^{k}\left(-1\right)^{b}\prod_{i=1}^{b}\tau_{i,n-1}\sum_{j=1}^{n-1}\eta^{\left(j-1\right)s}x_{v_{1},\cdots,v_{j-1},v_{j}',v_{j+1},\cdots,v_{n-1}}^{1-b+k,\cdots,j-1-b+k,j-b+k,j+1-b+k,\cdots,n-1-b+k},
\end{alignat*}
where $v_{j}'=\text{Res}_{z}Y\left(u,z\right)v_{j}\frac{\left(1+z\right)^{\text{wt}u-1+\frac{s}{k}}}{z}.$
For $1\le t\le k-n+1,$ $1\le s\le k-1$, set
\[
y_{t}=\sum_{b=1}^{k}\left(-1\right)^{b}g^{b-1}x_{u,v_{1},\cdots,v_{n-1}}^{n-1+t,1,\cdots,n-1}
\]
and
\[
u_{s}=-\sum_{b=1}^{k}\left(-1\right)^{b}\prod_{i=1}^{b}\tau_{i,n-1}\eta^{\left(j-1\right)s}x_{v_{1},\cdots,v_{j-1},v_{j}',v_{j+1},\cdots,v_{n-1}}^{1-b+k,\cdots,j-1-b+k,j-b+k,j+1-b+k,\cdots,n-1-b+k}.
\]
Then
\begin{alignat*}{1}
 & \sum_{b=1}^{k}\sum_{2-a\notin\left\{ 1-b,2-b,\cdots,n-1-b\right\} }\left(-1\right)^{b}\prod_{i=1}^{b}\tau_{i,n-1}\eta^{\left(b-a+1\right)s}x_{u,v_{1},\cdots,v_{n-1}}^{2-a+k,1-b+k,2-b+k,\cdots,n-1-b+k}\\
= & \sum_{j=1}^{k-n+1}\eta^{(n+j-2)s}y_{j}\equiv u_{s}.
\end{alignat*}
It is clear that the linear system has a solution such that each $n$-tensor
vector $y_{t}\in\left(V^{\otimes k}\right)_{\bar{1}}^{0\ast}$ is
a linear combination of $\left(n-1\right)$-tensor vector $u_{s}\in\left(V^{\otimes k}\right)_{\bar{1}}$
for all $t$. In particular, when $n=2$,
\begin{gather}
y_{t}=\sum_{b=1}^{k}\left(-1\right)^{b}g^{b-1}x_{u,v_{1}}^{1+t,1}=\sum_{a=1}^{k}\left(-1\right)^{a}x_{u,v_{1}}^{a+t,a},t=1,\cdots,k-1;u\in V_{\bar{0}},v_{1}\in V_{\bar{1}};\nonumber \\
u_{s}=-\sum_{b=1}^{k}\left(-1\right)^{b}x_{v'_{1}}^{2-b+k}=-\sum_{a=1}^{k}\left(-1\right)^{a}x_{v'_{1}}^{a},s=1,\cdots,k-1;v'_{1}\in V_{\bar{1}}.\label{U_s  odd even}
\end{gather}

(ii) Let $u^{1}\in\left(V^{\otimes k}\right)_{\bar{1}}$ and $x_{v_{1},\cdots,v_{n-1}}^{1,\cdots,n-1}\in\left(V^{\otimes k}\right)_{\bar{0}}.$
Then $\sigma\left(u^{1}\right)=-u^{1}$ and $\sigma\left(x_{v_{1},\cdots,v_{n-1}}^{1,\cdots,n-1}\right)=x_{v_{1},\cdots,v_{n-1}}^{1,\cdots,n-1}.$
In this case, we have

\[
\sum_{a=1}^{k}\eta^{-\left(a-1\right)s}h^{a-1}u^{1}=\sum_{a=1}^{k}\eta^{-\left(a-1\right)s}\left(-1\right)^{a-1}u^{k-a+2}
\]
and
\begin{alignat*}{1}
\sum_{b=1}^{k}\eta^{\left(b-1\right)s}h^{b-1}x_{v_{1},\cdots,v_{n-1}}^{1,\cdots,n-1} & =\sum_{b=1}^{k}\eta^{bs}\prod_{i=1}^{b}\tau_{i,n-1}x_{v_{1},v_{2},\cdots,v_{n-1}}^{1-b+k,2-b+k,\cdots,n-1-b+k},
\end{alignat*}
where $m-b+k$ is understood to be $m-b$ if $m-b>0$ and $\tau_{i,n-1}$
is defined in (\ref{sign function}). Then we obtain

\begin{alignat*}{1}
0 & \equiv\left(\sum_{a=1}^{k}\eta^{-\left(a-1\right)s}h^{a-1}u^{1}\right)\circ_{g}\left(\sum_{b=1}^{k}\eta^{\left(b-1\right)s}h^{b-1}x_{v_{1},\cdots,v_{n-1}}^{1,\cdots,n-1}\right)\\
 & =\left(\sum_{a=1}^{k}\eta^{-\left(a-1\right)s}\left(-1\right)^{a-1}u^{2-a+k}\right)\circ_{g}\left(\sum_{b=1}^{k}\eta^{bs}\prod_{i=1}^{b}\tau_{i,n-1}x_{v_{1},v_{2},\cdots,v_{n-1}}^{1-b+k,2-b+k,\cdots,n-1-b+k}\right)\\
 & =\text{Res}_{z}Y\left(\sum_{a=1}^{k}\eta^{-\left(a-1\right)s}\left(-1\right)^{a-1}u^{2-a+k},z\right)\\
 & \cdot\left(\sum_{b=1}^{k}\eta^{bs}\prod_{i=1}^{b}\tau_{i,n-1}x_{v_{1},v_{2},\cdots,v_{n-1}}^{1-b+k,2-b+k,\cdots,n-1-b+k}\right)\frac{\left(1+z\right)^{\text{wt}u-1+\frac{s}{k}}}{z}\\
 & =\sum_{a=1}^{k}\sum_{b=1}^{k}\left(-1\right)^{a-1}\prod_{i=1}^{b}\tau_{i,n-1}\eta^{\left(b-a+1\right)s}\text{Res}_{z}Y\left(u^{2-a+k},z\right)\left(x_{v_{1},v_{2},\cdots,v_{n-1}}^{1-b+k,2-b+k,\cdots,n-1-b+k}\right)\frac{\left(1+z\right)^{\text{wt}u-1+\frac{s}{k}}}{z}\\
 & =\sum_{b=1}^{k}\prod_{i=1}^{b}\tau_{i,n-1}\sum_{2-a\notin\left\{ 1-b,2-b,\cdots,n-1-b\right\} }\left(-1\right)^{a-1}\eta^{\left(b-a+1\right)s}\epsilon_{a}x_{u,v_{1},v_{2},\cdots,v_{n-1}}^{2-a+k,1-b+k,2-b+k,\cdots,n-1-b+k}\\
 & +\sum_{b=1}^{k}\prod_{i=1}^{b}\tau_{i,n-1}\sum_{j=1}^{n-1}\left(-1\right)^{b-j}\eta^{\left(j-1\right)s}\epsilon_{j}x_{v_{1},\cdots,v_{j-1},v_{j}',v_{j+1},\cdots,v_{n-1}}^{1-b+k,\cdots,j-1-b+k,j-b+k,j+1-b+k,\cdots,n-1-b+k},
\end{alignat*}
where $v_{j}'=\text{Res}_{z}Y\left(u,z\right)v_{j}\frac{\left(1+z\right)^{\text{wt}u-1+\frac{s}{k}}}{z}$,
and $\epsilon_{a},\epsilon_{j}=\pm1$ are determined by (\ref{tensor product-super sign}).
For $1\le t\le k-n+1,$ $1\le s\le k-1$, set
\begin{align*}
y_{t} & =\epsilon_{t}\sum_{b=1}^{k}\left(-1\right)^{1-n-t+b}g^{b-1}x_{u,v_{1},\cdots,v_{n-1}}^{n-1+t,1,\cdots,n-1}
\end{align*}
and
\[
u_{s}=-\sum_{b=1}^{k}\prod_{i=1}^{b}\tau_{i,n-1}\sum_{j=1}^{n-1}\left(-1\right)^{b-j}\eta^{\left(j-1\right)s}\epsilon_{j}x_{v_{1},\cdots,v_{j-1},v_{j}',v_{j+1},\cdots,v_{n-1}}^{1-b+k,\cdots,j-1-b+k,j-b+k,j+1-b+k,\cdots,n-1-b+k},
\]
where $\epsilon_{t}$ is determined by (\ref{tensor product-super sign}).
Then
\begin{alignat*}{1}
 & \sum_{b=1}^{k}\sum_{2-a\notin\left\{ 1-b,2-b,\cdots,n-1-b\right\} }\left(-1\right)^{a-1}\epsilon_{a}\prod_{i=1}^{b}\tau_{i,n-1}\eta^{\left(b-a+1\right)s}x_{u,v_{1},v_{2},\cdots,v_{n-1}}^{1-b+k,2-b+k,\cdots,n-1-b+k}\\
= & \sum_{j=1}^{k-n+1}\eta^{(n+j-2)s}y_{j}\equiv u_{s}.
\end{alignat*}
It is clear that the linear system has a solution such that each $n$-tensor
vector $y_{t}$ is a linear combination of $\left(n-1\right)$-tensor
vector $u_{s}$ for all $t$. In particular, when $n=2$,
\begin{alignat}{1}
 & y_{t}=\left(-1\right)^{t}\sum_{b=1}^{k}\left(-1\right)^{b-1}g^{b-1}x_{u,v_{1}}^{1+t,1}=\left(-1\right)^{t}\sum_{a=1}^{k}\left(-1\right)^{a-1}x_{u,v_{1}}^{a+t,a},t=1,\cdots,k-1;u\in V_{\bar{1}},v_{1}\in V_{\bar{0}};\nonumber \\
 & u_{s}=-\sum_{b=1}^{k}\left(-1\right)^{b-1}x_{v'_{1}}^{2-b+k}=-\sum_{a=1}^{k}\left(-1\right)^{a-1}x_{v'_{1}}^{a},s=1,\cdots,k-1;v'_{1}\in V_{\bar{1}}.\label{U_s even odd}
\end{alignat}

Combining (i) and (ii), by induction on $n$ we obtain that any $n$-tensor
vector in $\left(V^{\otimes k}\right)_{\bar{1}}^{0\ast}$ can be reduced
to vectors of the form $\sum_{b=1}^{k}\left(-1\right)^{b-1}v^{b}=\left[v^{1}-v^{2}+v^{3}-\cdots+v^{k-1}-v^{k}\right]$
for $v\in V_{\bar{1}}.$

Therefore, $A_{g}\left(V^{\otimes k}\right)$ is spanned by $[\sum_{a=1}^{k}v^{a}]\text{ for \ensuremath{v\in V_{\bar{0}}}}$
and $[\sum_{a=1}^{k}\left(-1\right)^{a-1}v^{a}]$ for $v\in V_{\bar{1}}.$
\end{proof}
\begin{remark} \label{remark} For $n=2$ and $j=1,\cdots,k-1$,
we will need an explicit expression of $y_{j}$ in terms of $u_{s}$
for later purpose. Let $n=2$ in the proof of Lemma \ref{n-tensor to n-1-tensor-1-even},
we obtain $\sum_{j=1}^{k-1}\eta^{js}y_{j}=u_{s}$. Applying \cite[Lemma 3.2]{DXY},
we obtain that $y_{j}=\sum_{s=1}^{k-1}\frac{1}{k}\left(\eta^{-js}-1\right)u_{s}$
with $u_{s}$ given by (\ref{U_s even even}), (\ref{U_s odd odd}),
(\ref{U_s  odd even}) and (\ref{U_s even odd}).

\end{remark}

Now we prove the main theorem.

\begin{theorem}\label{mainthm-even} Let $g=\left(1\ 2\cdots k\right)$
with $k$ even. Define
\begin{eqnarray*}
\phi:A_{g}\left(V^{\otimes k}\right) & \to & A_{\sigma}\left(V\right)\\{}
[\sum_{a=1}^{k}u^{a}] & \mapsto & [k\Delta_{k}\left(1\right)u],\ u\in V_{\bar{0}},\\{}
[\sum_{a=1}^{k}\left(-1\right)^{a-1}v^{a}] & \mapsto & [k\Delta_{k}\left(1\right)v],\ v\in V_{\bar{1}}.
\end{eqnarray*}
Then $\phi$ gives an isomorphism between $A_{g}\left(V^{\otimes k}\right)$
and $A_{\sigma}\left(V\right).$\end{theorem}
\begin{proof}
Let $h=g\sigma$. Then $o\left(h\right)=k$. Recall that $A_{g}\left(V^{\otimes k}\right)=V^{\otimes k}/O_{g}\left(V^{\otimes k}\right)$
and $A_{\sigma}\left(V\right)=V/O_{\sigma}\left(V\right)$ where $O_{g}\left(V^{\otimes k}\right)$
and $O_{\sigma}\left(V\right)$ are defined in Section \ref{sec:Basics}.
To show that $\phi$ is an isomorphism between the associative algebras
$A_{g}\left(V^{\otimes k}\right)$ and $A_{\sigma}\left(V\right)$,
we need to show that $\phi$ is well-defined and it is a homomorphism.

First we show that $\phi$ is well-defined. Since $A_{g}\left(V^{\otimes k}\right)=\left(V^{\otimes k}\right){}^{0\ast}/O_{g}\left(V^{\otimes k}\right)\cap\left(V^{\otimes k}\right){}^{0\ast}$
and $O_{g}\left(V^{\otimes k}\right)\cap\left(V^{\otimes k}\right){}^{0\ast}$
is spanned by $u\circ_{g}v$ for $u\in\left(V^{\otimes k}\right){}^{s\ast}$
and $v\in\left(V^{\otimes k}\right){}^{\left(k-s\right)\ast}$, we
map $u\circ_{g}v$ to $0$ for $s=1,...,k-1.$ In fact, from the proof
of Lemma \ref{n-tensor to n-1-tensor-1-even} we see that $\left[u\circ_{g}v\right]=0$
just gives an identification between $p$-tensor vectors and $q$-tensor
vectors in $A_{g}\left(V^{\otimes k}\right)$ such that either $p>1$
or $q>1.$ It suffices to prove that for any $x,y\in\left(V^{\otimes k}\right)^{0\ast}$,
$\phi\left(x\circ_{g}y\right)$$\in O_{\sigma}\left(V\right).$

For convenience, in the following we will denote $\overline{u}=\sum_{a=1}^{k}u^{a}$
for $u\in V_{\bar{0}}$ and $\widetilde{u}=\sum_{a=1}^{k}\left(-1\right)^{a-1}u^{a}$
for $u\in V_{\bar{1}}.$

(1) Let $\overline{u}=\sum_{a=1}^{k}u^{a},\overline{v}=\sum_{b=1}^{k}v^{b}\in\left(V^{\otimes k}\right)_{\bar{0}}^{0\ast}$
with $u,v\in V_{\bar{0}}.$ Then we have
\[
\overline{u}\circ_{g}\overline{v}=\text{Res}_{z}\sum_{a,b=1}^{k}Y\left(u^{a},z\right)v^{b}\frac{\left(1+z\right)^{\text{wt}u}}{z^{2}}=\sum_{j=1}^{k}\left(u\circ_{\sigma}v\right)^{j}+\text{wt}u\sum_{j=1}^{k-1}x_{j}+\sum_{j=1}^{k-1}y_{j},
\]
where
\[
x_{j}=\sum_{m=0}^{k-1}g^{m}x_{u,v}^{1,1+j}=\sum_{m=0}^{k-1}x_{u,v}^{1-m+k,1+j-m+k}\equiv\sum_{t=1}^{k-1}\frac{1}{k}\left(\eta^{-jt}-1\right)u_{t},
\]
\[
\ y_{j}=\sum_{m=0}^{k-1}g^{m}x_{u_{-2}\textbf{1},v}^{1,j}=\sum_{m=0}^{k-1}x_{u_{-2}\textbf{1},v}^{1-m+k,1+j-m+k}\equiv\sum_{t=1}^{k-1}\frac{1}{k}\left(\eta^{-jt}-1\right)w_{t}
\]
with
\[
u_{t}=-\text{Res}_{z}\sum_{j=1}^{k}\left(Y\left(u,z\right)v\frac{\left(1+z\right)^{\text{wt}u-1+\frac{t}{k}}}{z}\right)^{j},
\]
and
\[
w_{t}=-\text{Res}_{z}\sum_{j=1}^{k}\left(Y\left(u_{-2}\textbf{1},z\right)v\frac{\left(1+z\right)^{\text{wt}u+\frac{t}{k}}}{z}\right)^{j}
\]
where we use Remark \ref{remark}. By the same arguments in \cite[Theorem 3.7]{DXY},
we obtain
\begin{align*}
 & \phi\left(\overline{u}\circ_{g}\overline{v}\right)\\
 & =k\Delta_{k}\left(1\right)u\circ_{\sigma}v-\text{wt}u\sum_{t,j=1}^{k-1}\frac{1}{k}\left(\eta^{-jt}-1\right)k\Delta_{k}\left(1\right)\text{Res}_{z}Y\left(u,z\right)v\frac{\left(1+z\right)^{\text{wt}u-1+\frac{t}{k}}}{z}\\
 & -\sum_{t,j=1}^{k-1}\frac{1}{k}\left(\eta^{-jt}-1\right)k\Delta_{k}\left(1\right)\text{Res}_{z}Y\left(u_{-2}\textbf{1},z\right)v\frac{\left(1+z\right)^{\text{wt}u+\frac{t}{k}}}{z}\\
 & =k^{1-\text{wt}u}\text{Res}_{z}Y\left(e^{\sum_{j\in\mathbb{Z}_{+}}a_{j}(1+z)^{-j}L\left(j\right)}u,z\right)\Delta_{k}\left(1\right)v\left(1+z\right)^{\text{wt}u}\frac{1}{z^{2}}\in O_{\sigma}\left(V\right).
\end{align*}
Thus we obtain $\phi\left(\overline{u}\circ_{g}\overline{v}\right)=\lambda u\circ_{\sigma}v\in O_{\sigma}\left(V\right)$
for some constant $\lambda.$

(2) Let $\overline{u}=\sum_{a=1}^{k}u^{a}\in\left(V^{\otimes k}\right)_{\bar{0}}^{0\ast}$,
$\widetilde{v}=\sum_{b=1}^{k}\left(-1\right)^{b-1}v^{b}\in\left(V^{\otimes k}\right)_{\bar{1}}^{0\ast}$
with $u\in V_{\bar{0}}$ and $v\in V_{\bar{1}}.$ Then
\[
\overline{u}\circ_{g}\widetilde{v}=\text{Res}_{z}\sum_{a,b=1}^{k}Y\left(u^{a},z\right)\left(-1\right)^{b-1}v^{b}\frac{\left(1+z\right)^{\text{wt}u}}{z^{2}}=\sum_{j=1}^{k}\left(-1\right)^{j-1}\left(u\circ_{\sigma}v\right)^{j}+\text{wt}u\sum_{j=1}^{k-1}x_{j}+\sum_{j=1}^{k-1}y_{j},
\]
where
\[
x_{j}=\sum_{m=0}^{k-1}h^{m}x_{u,v}^{1,1+j}=\sum_{m=0}^{k-1}\left(-1\right)^{j-m}x_{u,v}^{1-m+k,1+j-m+k}\equiv\sum_{t=1}^{k-1}\frac{1}{k}\left(\eta^{-jt}-1\right)u_{t},
\]
\[
y_{j}=\sum_{m=0}^{k-1}h^{m}x_{u_{-2}\textbf{1},v}^{1,1+j}=\sum_{m=0}^{k-1}\left(-1\right)^{j-m}x_{u_{-2}\textbf{1},v}^{1-m+k,1+j-m+k}\equiv\sum_{t=1}^{k-1}\frac{1}{k}\left(\eta^{-jt}-1\right)w_{t}
\]
with
\[
u_{t}=-\text{Res}_{z}\sum_{j=1}^{k}\left(-1\right)^{j-1}\left(Y\left(u,z\right)v\frac{\left(1+z\right)^{\text{wt}u-1+\frac{t}{k}}}{z}\right)^{j},
\]
and
\[
w_{t}=-\text{Res}_{z}\sum_{j=1}^{k}\left(-1\right)^{j-1}\left(Y\left(u_{-2}\textbf{1},z\right)v\frac{\left(1+z\right)^{\text{wt}u+\frac{t}{k}}}{z}\right)^{j}
\]
where we use Remark \ref{remark}. Then
\begin{align*}
 & \phi\left(\overline{u}\circ_{g}\widetilde{v}\right)\\
 & =k\Delta_{k}\left(1\right)u\circ_{\sigma}v-\text{wt}u\sum_{t,j=1}^{k-1}\frac{1}{k}\left(\eta^{-jt}-1\right)k\Delta_{k}\left(1\right)\text{Res}_{z}Y\left(u,z\right)v\frac{\left(1+z\right)^{\text{wt}u-1+\frac{t}{k}}}{z}\\
 & -\sum_{t,j=1}^{k-1}\frac{1}{k}\left(\eta^{-jt}-1\right)k\Delta_{k}\left(1\right)\text{Res}_{z}Y\left(u_{-2}\textbf{1},z\right)v\frac{\left(1+z\right)^{\text{wt}u+\frac{t}{k}}}{z}.
\end{align*}

By the same arguments in \cite[Theorem 3.7]{DXY}, we obtain
\[
\phi\left(\overline{u}\circ_{g}\widetilde{v}\right)=k^{1-\text{wt}u}\text{Res}_{z}Y\left(e^{\sum_{j\in\mathbb{Z}_{+}}a_{j}(1+z)^{-j}L\left(j\right)}u,z\right)\Delta_{k}\left(1\right)v\left(1+z\right)^{\text{wt}u}\frac{1}{z^{2}}\in O\left(V\right).
\]
Thus $\phi\left(\overline{u}\circ_{g}\widetilde{v}\right)=\lambda u\circ_{\sigma}v\in O_{\sigma}\left(V\right)$
for some constant $\lambda.$

(3) Let $\widetilde{u}=\sum_{a=1}^{k}\left(-1\right)^{a-1}u^{a}\in\left(V^{\otimes k}\right)_{\bar{1}}^{0\ast},\overline{v}=\sum_{b=1}^{k}v^{b}\in\left(V^{\otimes k}\right)_{\bar{0}}^{0\ast}$
with $u\in V_{\bar{1}}$ and $v\in V_{\bar{0}}$. Using (\ref{U_s even odd})
and \cite[Lemma 3.2]{DXY}, by direct calculation we obtain
\begin{align*}
\widetilde{u}\circ_{g}\overline{v} & =\text{Res}_{z}\sum_{a,b=1}^{k}\left(-1\right)^{a-1}Y\left(u^{a},z\right)v^{b}\frac{\left(1+z\right)^{\text{wt}u}}{z^{2}}\\
 & =\sum_{j=1}^{k}\left(-1\right)^{j-1}\left(u\circ_{\sigma}v\right)^{j}+\text{wt}u\sum_{j=1}^{k-1}x_{j}+\sum_{j=1}^{k-1}y_{j}
\end{align*}
where
\[
x_{j}=\sum_{m=0}^{k-1}\left(-1\right)^{1-m-j}x_{u,v}^{1-m+k,1+j-m+k}\equiv\sum_{t=1}^{k-1}\frac{1}{k}\left(\eta^{-jt}-1\right)u_{t},
\]
\[
y_{j}=\sum_{m=0}^{k-1}\left(-1\right)^{1-m-j}x_{u_{-2}\textbf{1},v}^{1-m+k,1+j-m+k}\equiv\sum_{t=1}^{k-1}\frac{1}{k}\left(\eta^{-jt}-1\right)w_{t}
\]
with
\[
u_{t}=-\text{Res}_{z}\sum_{j=1}^{k}\left(-1\right)^{j-1}\left(Y\left(u,z\right)v\frac{\left(1+z\right)^{\text{wt}u-1+\frac{t}{k}}}{z}\right)^{j},
\]
and
\[
w_{t}=-\text{Res}_{z}\sum_{j=1}^{k}\left(-1\right)^{j-1}\left(Y\left(u_{-2}\textbf{1},z\right)v\frac{\left(1+z\right)^{\text{wt}u+\frac{t}{k}}}{z}\right)^{j}.
\]
By the same arguments in \cite[Theorem 3.7]{DXY}, we obtain
\[
\phi\left(\widetilde{u}\circ_{g}\overline{v}\right)=k^{1-\text{wt}u}\text{Res}_{z}Y\left(e^{\sum_{j\in\mathbb{Z}_{+}}a_{j}(1+z)^{-j}L\left(j\right)}u,z\right)\Delta_{k}\left(1\right)v\left(1+z\right)^{\text{wt}u}\frac{1}{z^{2}}\in O_{\sigma}\left(V\right).
\]
Thus $\phi\left(\widetilde{u}\circ_{g}\overline{v}\right)=\lambda u\circ_{\sigma}v\in O_{\sigma}\left(V\right)$
for some constant $\lambda.$

(4) Let $\widetilde{u}=\sum_{a=1}^{k}\left(-1\right)^{a+1}u^{a},\widetilde{v}=\sum_{b=1}^{k}\left(-1\right)^{b+1}v^{b}\in\left(V^{\otimes k}\right)_{\bar{1}}^{0\ast}$
with $u,v\in V_{\bar{1}}$. By using (\ref{U_s odd odd}) and \cite[Lemma 3.2]{DXY},
we obtain
\begin{align*}
\widetilde{u}\circ_{g}\widetilde{v} & =\text{Res}_{z}\sum_{a,b=1}^{k}\left(-1\right)^{a-1}Y\left(u^{a},z\right)\left(-1\right)^{b-1}v^{b}\frac{\left(1+z\right)^{\text{wt}u}}{z^{2}}\\
 & =\sum_{j=1}^{k}\left(u\circ_{\sigma}v\right)^{j}+\text{wt}u\sum_{j=1}^{k-1}x_{j}+\sum_{j=1}^{k-1}y_{j}
\end{align*}
where
\[
x_{j}=\sum_{a=1}^{k}\left(-1\right)^{j}x_{u,v}^{a,a+j}\equiv\sum_{t=1}^{k-1}\frac{1}{k}\left(\eta^{-jt}-1\right)u_{t},
\]
\[
y_{j}=\sum_{m=0}^{k-1}\left(-1\right)^{j}x_{u_{-2}\textbf{1},v}^{a,a+j}\equiv\sum_{t=1}^{k-1}\frac{1}{k}\left(\eta^{-jt}-1\right)w_{t}
\]
with
\[
u_{t}=-\text{Res}_{z}\sum_{j=1}^{k}\left(Y\left(u,z\right)v\frac{\left(1+z\right)^{\text{wt}u-1+\frac{t}{k}}}{z}\right)^{j},
\]
and
\[
w_{t}=-\text{Res}_{z}\sum_{j=1}^{k}\left(Y\left(u_{-2}\textbf{1},z\right)v\frac{\left(1+z\right)^{\text{wt}u+\frac{t}{k}}}{z}\right)^{j}.
\]
Then
\begin{align*}
 & \phi\left(\widetilde{u}\circ_{g}\widetilde{v}\right)\\
 & =k\Delta_{k}\left(1\right)u\circ v-\text{wt}u\sum_{t,j=1}^{k-1}\frac{1}{k}\left(\eta^{-jt}-1\right)k\Delta_{k}\left(1\right)\text{Res}_{z}Y\left(u,z\right)v\frac{\left(1+z\right)^{\text{wt}u-1+\frac{t}{k}}}{z}\\
 & -\sum_{t,j=1}^{k-1}\frac{1}{k}\left(\eta^{-jt}-1\right)k\Delta_{k}\left(1\right)\text{Res}_{z}Y\left(u_{-2}\textbf{1},z\right)v\frac{\left(1+z\right)^{\text{wt}u+\frac{t}{k}}}{z}.
\end{align*}

By the same arguments in \cite[Theorem 3.7]{DXY}, we obtain
\[
\phi\left(\widetilde{u}\circ_{g}\widetilde{v}\right)=k^{1-\text{wt}u}\text{Res}_{z}Y\left(e^{\sum_{j\in\mathbb{Z}_{+}}a_{j}(1+z)^{-j}L\left(j\right)}u,z\right)\Delta_{k}\left(1\right)v\left(1+z\right)^{\text{wt}u}\frac{1}{z^{2}}\in O_{\sigma}\left(V\right)
\]
and hence $\phi\left(\widetilde{u}\circ_{g}\widetilde{v}\right)=\lambda u\circ_{\sigma}v\in O_{\sigma}\left(V\right)$
for some constant $\lambda.$

Combining (1)-(4), we see that for any $x,y\in\left(V^{\otimes k}\right)^{0\ast}$,
$\phi\left(x\circ_{g}y\right)$$\in O_{\sigma}\left(V\right).$

Now we prove that $\phi$ is a homomorphism.

(i) For $\overline{u}=\sum_{a=1}^{k}u^{a},\overline{v}=\sum_{b=1}^{k}v^{b}\in\left(V^{\otimes k}\right)_{\bar{0}}^{0\ast}$
with $u,v\in V_{\bar{0}}$, we have
\begin{align*}
\overline{u}\ast_{g}\overline{v} & =\text{Res}_{z}Y\left(\sum_{a=1}^{k}u^{a},z\right)\left(\sum_{b=1}^{k}v^{b}\right)\frac{\left(1+z\right)^{\text{wt}u}}{z}\\
 & =\sum_{a\not=b}x_{u,v}^{a,b}+\sum_{j=1}^{k}\left(u\ast_{\sigma}v\right)^{j}\\
 & =\sum_{b-a=j,j=1,\cdots,k-1}x_{u,v}^{a,b}+\sum_{j=1}^{k}\left(u\ast_{\sigma}v\right)^{j}\\
 & =\sum_{a=1}^{k}\sum_{j=1}^{k-1}x_{u,v}^{a,a+j}+\sum_{j=1}^{k}\left(u\ast_{\sigma}v\right)^{j}\\
 & =\sum_{j=1}^{k-1}x_{j}+\sum_{j=1}^{k}\left(u\ast_{\sigma}v\right)^{j},
\end{align*}
where $x_{j}=\sum_{a=1}^{k}x_{u,v}^{a,a+j}$. By the same arguments
in \cite[Theorem 3.7]{DXY}, we obtain
\begin{align*}
 & \phi\left(\overline{u}\ast_{g}\overline{v}\right)\\
= & -\frac{1}{k}\sum_{t,j=1}^{k-1}\left(\eta^{-jt}-1\right)k\Delta_{k}\left(1\right)\text{Res}_{z}Y\left(u,z\right)v\frac{\left(1+z\right)^{\text{wt}u-1+\frac{t}{k}}}{z}+k\Delta_{k}\left(1\right)\left(u\ast_{\sigma}v\right)\\
= & k^{2}\text{Res}_{z}Y\left(\Delta_{k}\left(1\right)u,z\right)\Delta_{k}\left(1\right)v\left(1+z\right)^{\text{wt}v}\frac{1}{z}\\
= & \left(k\Delta_{k}\left(1\right)u\right)\ast_{\sigma}\left(k\Delta_{k}\left(1\right)v\right)\\
= & \phi\left(\overline{u}\right)\ast_{\sigma}\phi\left(\overline{v}\right).
\end{align*}

(ii) For $\overline{u}=\sum_{a=1}^{k}u^{a}\in\left(V^{\otimes k}\right)_{\bar{0}}^{0\ast}$,
$\widetilde{v}=\sum_{b=1}^{k}\left(-1\right)^{b-1}v^{b}\in\left(V^{\otimes k}\right)_{\bar{1}}^{0\ast}$
with $u\in V_{\bar{0}}$ and $v\in V_{\bar{1}}$. We have
\begin{align*}
\phi\left(\overline{u}\ast_{g}\widetilde{v}\right) & =\text{Res}_{z}Y\left(\sum_{a=1}^{k}u^{a},z\right)\left(\sum_{b=1}^{k}\left(-1\right)^{b-1}v^{b}\right)\frac{\left(1+z\right)^{\text{wt}u}}{z}\\
 & =\sum_{a\not=b}\left(-1\right)^{b-1}x_{u,v}^{a,b}+\sum_{j=1}^{k}\left(-1\right)^{j-1}\left(u\ast_{\sigma}v\right)^{j}\\
 & =\sum_{b-a=j,j=1,\cdots,k-1}\left(-1\right)^{b-1}x_{u,v}^{a,b}+\sum_{j=1}^{k}\left(-1\right)^{j-1}\left(u\ast_{\sigma}v\right)^{j}\\
 & =\sum_{a=1}^{k}\sum_{j=1}^{k-1}\left(-1\right)^{a+j-1}x_{u,v}^{a,a+j}+\sum_{j=1}^{k}\left(-1\right)^{j-1}\left(u\ast_{\sigma}v\right)^{j}\\
 & =\sum_{j=1}^{k-1}y_{j}+\sum_{j=1}^{k}\left(-1\right)^{j-1}\left(u\ast_{\sigma}v\right)^{j},
\end{align*}
where $y_{j}=\sum_{a=1}^{k}\left(-1\right)^{a+j-1}x_{u,v}^{a,a+j}$.
By the same arguments in \cite[Theorem 3.7]{DXY}, we obtain
\begin{align*}
 & \phi\left(\overline{u}\ast_{g}\widetilde{v}\right)\\
= & -\frac{1}{k}\sum_{t,j=1}^{k-1}\left(\eta^{-jt}-1\right)k\Delta_{k}\left(1\right)\text{Res}_{z}Y\left(u,z\right)v\frac{\left(1+z\right)^{\text{wt}u-1+\frac{t}{k}}}{z}+k\Delta_{k}\left(1\right)\left(u\ast_{\sigma}v\right)\\
= & k^{2}\text{Res}_{z}Y\left(\Delta_{k}\left(1\right)u,z\right)\Delta_{k}\left(1\right)v\left(1+z\right)^{\text{wt}v}\frac{1}{z}\\
= & \left(k\Delta_{k}\left(1\right)u\right)\ast_{\sigma}\left(k\Delta_{k}\left(1\right)v\right)\\
= & \phi\left(\overline{u}\right)\ast_{\sigma}\phi\left(\widetilde{v}\right).
\end{align*}

(iii) For $\widetilde{u}=\sum_{a=1}^{k}\left(-1\right)^{a-1}u^{a}\in\left(V^{\otimes k}\right)_{\bar{1}}^{0\ast},\overline{v}=\sum_{b=1}^{k}v^{b}\in\left(V^{\otimes k}\right)_{\bar{0}}^{0\ast}$
with $u\in V_{\bar{1}}$ and $v\in V_{\bar{0}}$, we have
\begin{align*}
\widetilde{u}\ast_{g}\overline{v} & =\text{Res}_{z}Y\left(\sum_{a=1}^{k}\left(-1\right)^{a-1}u^{a},z\right)\left(\sum_{b=1}^{k}v^{b}\right)\frac{\left(1+z\right)^{\text{wt}u}}{z}\\
 & =\sum_{a\not=b}\left(-1\right)^{a-1}x_{u,v}^{a,b}+\sum_{j=1}^{k}\left(-1\right)^{j-1}\left(u\ast_{\sigma}v\right)^{j}\\
 & =\sum_{b-a=j,j=1,\cdots,k-1}\left(-1\right)^{a-1}x_{u,v}^{a,b}+\sum_{j=1}^{k}\left(-1\right)^{j-1}\left(u\ast_{\sigma}v\right)^{j}\\
 & =\sum_{a=1}^{k}\sum_{j=1}^{k-1}\left(-1\right)^{a-1-j}x_{u,v}^{a,a+j}+\sum_{j=1}^{k}\left(-1\right)^{j-1}\left(u\ast_{\sigma}v\right)^{j}\\
 & =\sum_{j=1}^{k-1}y_{j}+\sum_{j=1}^{k}\left(-1\right)^{j-1}\left(u\ast_{\sigma}v\right)^{j},
\end{align*}
where $y_{j}=\sum_{a=1}^{k}\left(-1\right)^{a-1-j}x_{u,v}^{a,a+j}$.
Using (\ref{U_s even odd}) and \cite[Lemma 3.2]{DXY}, we obtain
\[
\phi\left(\widetilde{u}\ast_{g}\overline{v}\right)=k\Delta_{k}\left(1\right)\text{Res}_{z}Y\left(u,z\right)v\left(1+z\right)^{\text{wt}u}\left(\frac{1}{z}+\sum_{t=1}^{k-1}\frac{\left(1+z\right)^{\text{wt}u-1+\frac{t}{k}}}{z}\right).
\]
By the same arguments in \cite[Theorem 3.7]{DXY}, we obtain $\phi\left(\widetilde{u}\ast_{g}\overline{v}\right)=\phi\left(\widetilde{u}\right)\ast_{\sigma}\phi\left(\overline{v}\right)$.

(iv) For $\widetilde{u}=\sum_{a=1}^{k}\left(-1\right)^{a+1}u^{a},\widetilde{v}=\sum_{b=1}^{k}\left(-1\right)^{b+1}v^{b}\in\left(V^{\otimes k}\right)_{\bar{1}}^{0\ast}$
with $u,v\in V_{\bar{1}}$, we have
\begin{align*}
\widetilde{u}\ast_{g}\widetilde{v} & =\text{Res}_{z}Y\left(\sum_{a=1}^{k}\left(-1\right)^{a-1}u^{a},z\right)\left(\sum_{b=1}^{k}\left(-1\right)^{b-1}v^{b}\right)\frac{\left(1+z\right)^{\text{wt}u}}{z}\\
 & =\sum_{a\not=b}\left(-1\right)^{a+b}x_{u,v}^{a,b}+\sum_{j=1}^{k}\left(u\ast_{\sigma}v\right)^{j}\\
 & =\sum_{b-a=j,j=1,\cdots,k-1}\left(-1\right)^{a+b}x_{u,v}^{a,b}+\sum_{j=1}^{k}\left(u\ast_{\sigma}v\right)^{j}\\
 & =\sum_{a=1}^{k}\sum_{j=1}^{k-1}\left(-1\right)^{j}x_{u,v}^{a,a+j}+\sum_{j=1}^{k}\left(u\ast_{\sigma}v\right)^{j}\\
 & =\sum_{j=1}^{k-1}y_{j}+\sum_{j=1}^{k}\left(u\ast_{\sigma}v\right)^{j},
\end{align*}
where $y_{j}=\sum_{a=1}^{k}\left(-1\right)^{j}x_{u,v}^{a,a+j}$. As
before, we have
\begin{align*}
 & \phi\left(\widetilde{u}\ast_{g}\widetilde{v}\right)\\
= & -\frac{1}{k}\sum_{t,j=1}^{k-1}\left(\eta^{-jt}-1\right)k\Delta_{k}\left(1\right)\text{Res}_{z}Y\left(u,z\right)v\frac{\left(1+z\right)^{\text{wt}u-1+\frac{t}{k}}}{z}+k\Delta_{k}\left(1\right)\left(u\ast_{\sigma}v\right).
\end{align*}
Then by the arguments in \cite[Theorem 3.7]{DXY}, one can prove $\phi\left(\widetilde{u}\ast_{g}\widetilde{v}\right)=\phi\left(\widetilde{u}\right)\ast_{\sigma}\phi\left(\widetilde{v}\right)$.

Combining (i)-(iv), we get that $\phi$ is a homomorphism.

Since $\Delta_{k}\left(1\right)$ is an invertible operator, we define
\begin{alignat*}{1}
\psi: & A_{\sigma}\left(V\right)\to A_{g}\left(V^{\otimes k}\right)\\
 & u\mapsto\frac{1}{k}\sum_{a=1}^{k}\left(\Delta_{k}\left(1\right)^{-1}u\right)^{a}.
\end{alignat*}
Thus $\psi$ is the inverse of $\phi$ and $\phi$ defines an isomorphism.
\end{proof}
For arbitrary element $g\in S_{k}$, one has $g=g_{1}\cdots g_{s}g_{s+1}\cdots g_{s+t}$
which is a product of disjoint cycles with $g_{1},\cdots,g_{s}$ odd
cycles and $g_{s+1},\cdots,g_{s+t}$ even cycles. If $U,W$ are vertex
operator algebras with automorphisms $f,h$ respectively, then $f\otimes h$
is an automorphism of $U\otimes W$ and $A_{f\otimes h}\left(U\otimes W\right)$
is isomorphic to $A_{f}(U)\otimes_{\mathbb{C}}A_{h}(W).$ Using this
fact and Theorems \ref{mainthm-odd} and \ref{mainthm-even}, we conclude
this paper with the following corollary.

{\begin{corollary} Let $g=g_{1}\cdots g_{s}g_{s+1}\cdots g_{s+t}\in S_{k}$
be a product of disjoint cycles with $g_{1},\cdots,g_{s}$ odd cycles
and $g_{s+1},\cdots,g_{s+t}$ even cycles. Then $A_{g}\left(V^{\otimes k}\right)$
is isomorphic to $\left(A\left(V\right)\right){}^{\otimes s}\otimes\left(A_{\sigma}\left(V\right)\right)^{\otimes t}$
where the tensor is over $\mathbb{C}.$ \end{corollary}

\textbf{\footnotesize{}{}\hspace{1.8em}C. Dong}{\footnotesize{}{}:
Department of Mathematics, University of California Santa Cruz, CA
95064 USA; }\texttt{dong@ucsc.edu}

\textbf{\footnotesize{}{}F. Xu}{\footnotesize{}{}: Department of
Mathematics, University of California, Riverside, CA 92521 USA; }\texttt{xufeng@math.ucr.edu}

\textbf{\footnotesize{}{}N. Yu}{\footnotesize{}{}: School of Mathematical
Sciences, Xiamen University, Fujian, 361005, CHINA;} \texttt{ninayu@xmu.edu.cn}
 }
\end{document}